\documentclass{amsart}

\RequirePackage{amsthm,amsmath,amsfonts,amssymb}
\RequirePackage[colorlinks,allcolors=blue]{hyperref}
\RequirePackage{graphicx}

\usepackage{natbib}
\usepackage[dvipsnames]{xcolor}
\usepackage{enumitem}
\usepackage{subcaption}
\usepackage{hyperref}
\usepackage[foot]{amsaddr}

\theoremstyle{plain}
\newtheorem{lemma}{Lemma}[section]
\newtheorem{theorem}{Theorem}[section]

\theoremstyle{remark}
\newtheorem{assumption}{Assumption}[section]
\newtheorem{example}{Example}[section]

\title[Weighted reduced rank estimators]{Weighted reduced rank estimators under cointegration rank uncertainty}
\author{Christian Holberg \and Susanne Ditlvesen}
\address{Department of Mathematical Sciences, University of Copenhagen, Universitetsparken 5, 2100
Copenhagen Ø, Denmark}
\email{c.holberg@math.ku.dk}
\email{susanne@math.ku.dk}

\begin{document}

\begin{abstract}
Cointegration analysis was developed for non-stationary linear processes that exhibit stationary relationships between coordinates. Estimation of the cointegration relationships in a multi-dimensional cointegrated process typically proceeds in two steps. First the rank is estimated, then the cointegration matrix is estimated, conditionally on the estimated rank (reduced rank regression). The asymptotics of the estimator is usually derived under the assumption of knowing the true rank. In this paper, we quantify the asymptotic bias and find the asymptotic distributions of the cointegration estimator in case of misspecified rank. Furthermore, we suggest a new class of weighted reduced rank estimators that allow for more flexibility in settings where rank selection is hard. We show empirically that a proper choice of weights can lead to increased predictive performance when there is rank uncertainty. Finally, we illustrate the estimators on empirical EEG data from a psychological experiment on visual processing.
\end{abstract}

\maketitle

\section{Introduction}

\subsection{Motivation}

Consider a $p$-dimensional autoregressive process $Y_t$ of order 1 (AR(1)) defined by a vector error correction model (VECM)  
\begin{equation}
\label{eq: vecm}\Delta Y_t = \Pi Y_{t-1} + Z_t
\end{equation}
where $\Delta Y = Y_t-Y_{t-1} \in \mathbb{R}^p$, $\Pi$ is the $p \times p$ autoregression matrix of fixed coefficients of rank $r \leq p$, and $Z_1, Z_2 \ldots $ are i.i.d. $p$-dimensional random vectors of mean zero. 


In standard low-dimensional problems, the typical procedure to determine $r$ is based on sequential likelihood-ratio tests \citep{johansen1995likelihood}. The test statistics do not follow any standard distributions, the critical values depend on $p$ and they need to be calculated numerically. Currently, critical values are available for dimension $p \leq 11$. This can be overcome by bootstrap methods. However, it is nontrivial to keep control over the type I error and the sequential testing can lead to severe bias, especially when the dimension $p$ of model \eqref{eq: vecm} increases \citep{ostergaard2022, onatski2018alternative}. Once the rank is fixed, a reduced rank regression \citep{anderson2002reduced} is performed assuming that this rank is in fact the true rank. 

In settings where rank estimation is hard or where the rank is fixed a priori, questions arise regarding properties of the estimator. One such question is how to characterize the asymptotic behaviour of reduced rank estimators where the rank is fixed at a value not necessarily equal to the true rank. Furthermore, treating the rank obtained from the sequential testing approach as fixed in the subsequent analysis neglects the added uncertainty. Thus, there is a  need for more flexible estimators that take the uncertainty into account. We suggest a weighted average of reduced rank estimators where ranks are weighted depending on the supporting evidence in the data. All the classical reduced rank estimators are special cases of this more general class of estimators.

\subsection{Literature Review}
For a review of reduced rank estimators with a fixed rank, using different types of penalizations, see \cite{levakova2023}.  

A lot of work has been done on the asymptotic behaviour of reduced rank estimators under the true rank or assuming wrongly full rank \(r=d\) (see, for example, \cite{johansen1988statistical, johansen1995likelihood, anderson2002reduced} for the time series setting and \cite{izenman1975reduced, anderson1999asymptotic} for the i.i.d setting). Less work has been done  under the assumption of a misspecified rank, that is, cases where the rank of the reduced rank estimator is not equal to the true rank of \(\Pi\). Only in the i.i.d. setting has any progress been made~\citep{anderson2002specification}.

Model averaging and weighted estimators for cointegrated VAR processes has also received little attention. See \cite{koop2006bayesian} for a review of Bayesian approaches. \cite{hansen2010averaging} deals with model averaging for one-dimensional processes with potential unit root. \cite{lieb2017inference} is most closely related to our approach, but they only consider one family of weights and their main concern is inference. 

\subsection{Our contribution}
There are three main contributions of the present paper. The first is to determine the asymptotic distribution of the reduced rank estimator of $\Pi$ under misspecified rank $r$. This has important statistical implications since there is no guarantee of determining even closely the true rank from finite sample sizes, especially for large $p$. We show that the cointegration estimator is consistent but has increased variance when the rank is overestimated, compared to a correctly specified rank. If the rank is underestimated, an asymptotic bias is introduced. The bias depends on the sizes of the eigenvalues of a certain eigenvalue problem. The main results are Theorem \ref{thm: over_rank} and \ref{thm: under_rank}. Especially the proof of Theorem \ref{thm: under_rank} is interesting. It relies on the delta method applied to a central limit theorem for a specific covariance matrix (Lemma \ref{lemma: xtil}) requiring us to carry out some novel computations involving matrix derivatives.

Our second contribution is the introduction of a new class of estimators which we call \emph{weighted reduced rank estimators}. We show how the classical reduced rank estimators are special cases of this class and how Theorem \ref{thm: over_rank} and \ref{thm: under_rank} determine the asymptotic behaviour for any particular weight \(w\in[0, 1]^p\). We argue why taking rank uncertainty into account is appropriate and show empirically in a simulation experiment how the predictive capabilities of the weighted reduced rank estimators outperform the classical reduced rank estimator based on pre-selected rank.

Our third contribution is the application of the new estimator on an experimental data set of Electroencephalography (EEG) measurements in a visual response study. We then compare the performance with the fixed rank estimators. In this study, $p=59$ with relatively small sample sizes, which is a typical setting where the rank is not well determined. We show that the smaller the sample size, the better the weighted reduced rank estimators perform compared to the fixed rank estimators, for any fixed rank, measured on mean square prediction error. This is important in many neurobiological studies, where the data dimension is high but sample size is restricted to a small time interval if a response to a stimulus is of interest.

\subsection{Organization}
The paper is organized as follows. In Section \ref{sec: pre} the model and assumptions for cointegration are presented. In Section \ref{sec: asymp} the asymptotic distribution under correctly specified rank is recalled and the main results are presented, namely the asymptotic distributions under misspecified rank. In Section \ref{sec: estim} we introduce the weighted reduced rank estimators. Section \ref{sec: simul} consists of two simulation experiments to verify our asymptotic results and compare the different estimators. Section \ref{sec: eeg} compares different weighted reduced rank estimators on a dataset of EEG signals and Section \ref{sec: con} concludes. All the proofs are presented in Section \ref{sec: proofs}. In the Appendix some auxiliary results are given. We show how the framework extends to processes of higher lags and give some further details regarding the simulations.  

\subsection{Notation}
$I_k$ denotes the $k$-dimensional identity matrix. Transposition is denoted by $^T$. Convergence in distribution is denoted by $\rightarrow_w$ and convergence in probability by $\rightarrow_p$. The Frobenius norm is denoted by \(||\cdot||_F\). For a matrix \(A\in\mathbb{R}^{n\times m}\) with \(n\ge m\) of full column rank \(m\), we write \(A_{\perp}\) to denote the \(n\times (n-m)\) matrix of full column rank \((n-m)\) such that \(\text{span}(A)^{\perp}=\text{span}(A_{\perp})\). If \(A\) is positive definite we write \(A^{\frac{1}{2}}\) for the unique positive definite matrix satisfying \(A^{\frac{1}{2}}A^{\frac{1}{2}} = A\). The vectorization operator is written as \(\text{vec}\).

\section{Preliminaries}
\label{sec: pre}

Let \(\{Y_t\}_{t=1}^{\infty}\) be defined by \eqref{eq: vecm}. Autoregressive processes of higher order, VAR(d) with $d>1$, are briefly treated in Appendix \ref{sec: multi}. 
Assume that \(Z_1, Z_2, ...\) are i.i.d. of mean zero, with covariance matrix \(\Sigma_Z := \mathbb{E}(Z_t Z_t^T)\), and a bounded fourth moment. Furthermore, assume that the process satisfies the usual cointegration assumption for some  \(0< r \le p\):

\begin{assumption} \label{as: roots}
The polynomial \(z\mapsto |(1 - z)I_p - \Pi z|\) has \(n = p - r\) unit roots and all other roots are outside the unit circle.
\end{assumption}

This assumption implies that the rank of \(\Pi\) is \(p - n = r\). Thus, we can decompose \(\Pi\) into two matrices \(\alpha, \beta \in \mathbb{R}^{p \times r}\) of rank \(r\) such that \(\Pi = \alpha \beta^T\). Let \(\alpha_{\perp}\) and \(\beta_{\perp}\) be orthogonal complements of \(\alpha\) and \(\beta\). This leads to the second condition that is usually assumed when working with cointegrated AR-processes \citep{johansen1995likelihood}.

\begin{assumption} \label{as: orth}
The $n\times n$ matrix \(\alpha_{\perp}^T \beta_{\perp}\) is non-singular.
\end{assumption}

Under these assumptions Granger's representation theorem \citep{englegranger1987} states that \(Y_t\) is integrated of order 1 (\(I(1)\)) and cointegrated of rank \(r\). The cointegration relations are given by \(\beta^T Y_t\). That is, \(Y_t\) exhibits random walk like behaviour with \(\Delta Y_t\) and \(\beta^T Y_t\) being stationary. Now define \(Q=(\beta, \alpha_{\perp})^T\) and note that 
\[
Q^{-1}=\left(
    \alpha(\beta^T \alpha)^{-1}, \beta_{\perp}(\alpha_{\perp}^T \beta_{\perp})^{-1}
\right).
\]
Then, with \(X_t = Q Y_t\), \(U_t = Q Z_t\), and 
\[
\Gamma = Q \Pi Q^{-1} = 
\begin{pmatrix}
    \beta^T \!\alpha  \, & \,   0   \\
    0               \, & \,   0
\end{pmatrix},
\]
we get the \(Q\)-transformed version of model (\ref{eq: vecm}),
\begin{equation}
    \Delta X_t = \Gamma X_{t-1} + U_t \label{eq: qvecm}.
\end{equation}
We have effectively split up the original process \(Y_t\) into a stationary part and a random walk part. In particular, if \(X_{1t}\) denotes the first \(r\) components of \(X_t\) and \(X_{2t}\) the last $n$ components, we have the following relations
\begin{align}
    \Delta X_{1t} &= \beta^T \alpha X_{1t-1} + U_{1t} \label{eq: xone}\\
    \Delta X_{2t} &= U_{2t}. \label{eq: xtwo}
\end{align}
We shall first study estimators of \(\Gamma\) from observations $X_0, X_1, \ldots , X_T$ and then transfer the results to the original parameter of interest, \(\Pi\). The reason for taking this small detour is that it will give more clarity to the limiting behaviour of different parts of the estimator corresponding to either the random walk or the stationary part of the process.

Before describing the asymptotics of the estimators we need some results regarding the cross-covariances. Specifically, define the empirical cross-covariances
\[
S_{X X} = \frac{1}{T}\sum_{t=1}^T X_{t-1} X_{t-1}^T, \quad
S_{U X} = \frac{1}{T}\sum_{t=1}^T U_t X_{t-1}^T,
\]
\[
S_{\Delta X X} = \frac{1}{T}\sum_{t=1}^T \Delta X X_{t-1}^T, \quad
S_{\Delta X \Delta X} = \frac{1}{T}\sum_{t=1}^T \Delta X \Delta X^T,
\]
and the covariance matrix \(\Sigma_U=\mathbb{E}(U_t U_t^T) = Q\Sigma_Z Q^T\). We use the following block matrix notation: For a \(p \times p\) matrix \(M\), let \(M_{11}\) denote the top left \(r\times r\) block, \(M_{22}\) the bottom right \(n\times n\) block, and \(M_{12}\) and \(M_{21}\) the two off-diagonal blocks. For notational convenience, we sometimes use a superscript instead. We implicitly assume that all the limits considered in the following sections are for \(T\rightarrow \infty\). For the stationary processes, \(X_{1t-1}\) and \(\Delta X_t\), the law of large numbers yields
\begin{align*}
    S^{11}_{XX} &\rightarrow_p \Sigma_X^{11} = \sum_{s=0}^{\infty} (I_r + \beta^T \alpha)^s \Sigma_U^{11} (I_r + \alpha^T \beta)^s \\
    S_{\Delta X \Delta X} &\rightarrow_p \Sigma_{\Delta X} = \begin{pmatrix}
        \Sigma_U^{11} + \beta^T\!\alpha \Sigma_X^{11} \alpha^T\!\beta \, & \,
        \Sigma_U^{12} \\
        \Sigma_U^{21} \, &
        \, \Sigma_U^{22}
    \end{pmatrix}.
\end{align*}
To study the asymptotics of the random walk part of the process we introduce a standard \(p\)-dimensional Brownian motion initiated at \(0\) denoted by \(\{W_s\}_{s\in[0,1]}\). We are now ready to present the crucial Lemma. A proof can be found in e.g. Lemma 7.1 in~\cite{lutkepohl2005new}.

\begin{lemma} \label{lemma: cross}
Define the \(n \times p\) matrix \(D = (0, I_n)\) and the random $r \times p$ matrix \(V^T:=(V_{11}^T, V_{21}^T)\) satisfying \(\textnormal{vec}V\sim \mathcal{N}(0, \Sigma_X^{11}\otimes \Sigma_U)\). The following converge jointly:
\begin{align}
    T^{-1}S_{XX}^{22} &\rightarrow_w D \Sigma_U^{\frac{1}{2}}\left(\int_0^1 W_s W_s^T\, ds \right) \Sigma_U^{\frac{1}{2}}D^T =: B \label{eq: cross_1}\\
    \begin{pmatrix}
        S_{UX}^{12} \\
        S_{UX}^{22}
    \end{pmatrix} &\rightarrow_w \Sigma_U^{\frac{1}{2}}\left(\int_0^1 W_s \, d W_s^T\right)^T \Sigma_U^{\frac{1}{2}} D^T =: \begin{pmatrix}
        J_{12} \\
        J_{22}
    \end{pmatrix} \label{eq: cross_2}\\
    T^{\frac{1}{2}}\begin{pmatrix}
        S_{UX}^{11} \\
        S_{UX}^{21}
    \end{pmatrix} &\rightarrow_w V = \begin{pmatrix}
        V_{11} \\
        V_{21}
    \end{pmatrix}. \label{eq: cross_3}
\end{align}
Furthermore,
\begin{align}
    S^{11}_{XX} &\rightarrow_p \Sigma_X^{11} \\
    S_{\Delta X \Delta X} &\rightarrow_p \Sigma_{\Delta X}.
\end{align}
\end{lemma}
A direct consequence of the above Lemma and summation by parts is that \(S_{\Delta X X}^{12}\rightarrow_p -\Sigma_U^{12}\) (see section 3.1 in \cite{anderson2002reduced}). Thus, we have fully uncovered the asymptotic behaviour of \(S_{\Delta X X}\) as well. Indeed, we can write
\[
\begin{pmatrix}
    S_{\Delta X X}^{11} \, & \, S_{\Delta X X}^{12} \\
    S_{\Delta X X}^{21} \, & \, S_{\Delta X X}^{22}
\end{pmatrix} = \begin{pmatrix}
    \beta^T\!\alpha S_{X X}^{11} \, & \, \beta^T\!\alpha S_{X X}^{12} \\
    0 & 0
\end{pmatrix} + \begin{pmatrix}
    S_{UX}^{11} \, & \, S_{UX}^{12} \\
    S_{UX}^{21} \, & \, S_{UX}^{22}
\end{pmatrix}.
\]
We have already established that the top right block converges in probability to \(-\Sigma_U^{12}\) and limits for the remaining three blocks follow easily from Lemma \ref{lemma: cross}. Note that, whereas \(S_{\Delta X X}^{11}\), \(S_{\Delta X X}^{12}\), and \(S_{\Delta X X}^{21}\) converge in probability, \(S_{\Delta X X}^{22}=S_{UX}^{22}\) converges only weakly. From this it also follows that \(S^{12}_{XX}\) and \(S^{21}_{XX}\) are bounded in probability whence \(T^{-\frac{1}{2}}S^{12}_{XX}, T^{-\frac{1}{2}}S^{21}_{XX}\rightarrow_p 0\).

\section{Asymptotic Distributions of Reduced Rank Estimators}
\label{sec: asymp}

With Lemma \ref{lemma: cross} in our arsenal, we are ready to study the asymptotic behaviour of estimators of $\Gamma$. In particular, we shall focus on the standard cointegration estimators \citep{johansen1995likelihood}. This is a collection of estimators that can be obtained by solving a generalized eigenvalue problem. We consider
\[
|S_{X \Delta X}(S_{\Delta X \Delta X})^{-1}S_{\Delta X X} - \hat{\lambda} S_{XX}|=0
\]
and order the solutions in decreasing order, \(\hat{\lambda}_1 \ge \hat{\lambda}_2 \ge ... \ge \hat{\lambda}_p\). With \(\hat{\Lambda}:=\textnormal{diag}(\hat{\lambda}_1, ..., \hat{\lambda}_p)\), denote by \(\hat{G}\) the \(p\times p\) matrix solving
\begin{equation}\label{eq: eig1}
     S_{X \Delta X}(S_{\Delta X \Delta X})^{-1}S_{\Delta X X} \hat{G} = S_{XX}\hat{G} \hat{\Lambda},  
\end{equation}
\begin{equation}\label{eq: eig2}
    \hat{G}^T S_{XX} \hat{G} = I_p.
\end{equation}
In column vector notation we write \(\hat{G} = (\hat{g}_1,..., \hat{g}_p)\). For any \(m_1\times m_2\) matrix \(M\) we shall write \(M^{:k}\) for the \(m_2\times k\) matrix consisting of the first \(k\le m_2\) columns of \(M\). Keeping in line with our previous block matrix notation, we write \(\hat{G}_{11}\) and \(\hat{G}_{22}\) for the top left \(r\times r\) block and the bottom right \(n\times n\) block of \(\hat{G}\) respectively and \(\hat{G}_{21}\) and \(\hat{G}_{12}\) for the off diagonal blocks. The reduced rank estimators are given by
\begin{equation}
\label{eq: Gammahatk}
\hat{\Gamma}_k = S_{\Delta X X}\hat{G}^{:k} \left(\hat{G}^{:k}\right)^T
\end{equation}
for \(k=1,...,p\). \(\hat{\Gamma}_k\) is called the reduced rank estimator of \(\Gamma\) for rank \(k\). One can show that \(\hat{\Gamma}_k\) is the maximum likelihood estimator of \(\Gamma\) under Gaussian errors when the data generating process is given by (\ref{eq: qvecm}) and the rank of \(\Gamma\) is fixed at \(k\)~\citep{johansen1995likelihood}. In our case the true rank is \(1 < r < p\) and the reduced rank estimator for a correctly specified rank is therefore \(\hat{\Gamma}_r\). Another special case is the least squares estimator \(\hat{\Gamma}_{\text{LS}}\) which is, in fact, equal to \(\hat{\Gamma}_p\).

For future reference, we also define the (appropriately rescaled) population versions of \(\hat{\lambda}\) and \(\hat{G}\). We let \(\lambda_1\ge \dots\ge \lambda_p\) be the ordered solutions to
\begin{equation} \label{eq: eig_pop}
    \left|\begin{pmatrix}
        \Sigma_X^{11}\alpha^T\beta\left(\Sigma_{\Delta X}^{-1}\right)_{11}\beta^T\alpha\Sigma_X^{11} & 0 \\
        0 & J_{22}
    \end{pmatrix} - \lambda \begin{pmatrix}
        \Sigma_X^{11} & 0 \\
        0 & B
    \end{pmatrix}\right| = 0
\end{equation}
with \(G=(g_1, \dots, g_p)\) the corresponding eigenvectors normalized so that 
\[
G^T\begin{pmatrix}
        \Sigma_X^{11} & 0 \\
        0 & B
    \end{pmatrix}G = I_p.
\]
The block diagonal structure implies that almost surely \(G_{12}\) and \(G_{21}\) are 0 with \(G_{11}\) and \(G_{22}\) solving the two seperate eigenvalue problems defined by the diagonal blocks in \eqref{eq: eig_pop}. This furthermore implies that the first \(r\) eigenvalues and eigenvectors are deterministic while the last \(n\) eigenvalues and eigenvectors are random. We let \(\Lambda = \text{diag}(\lambda_1, \dots, \lambda_p)\).

It makes sense to distinguish between three different situations and study them separately. First, the reduced rank estimator where the true rank is given a priori. In this case we include exactly enough information and the resulting estimator is optimal, among the estimators considered here, in the following sense: For all \(1\le k \le p\) for which \(\hat{\Gamma}_k\) is consistent, \(\hat{\Gamma}_r\) has the lowest asymptotic variance. 

Knowing the number of cointegrating relations, however, is often unrealistic. This leads us to consider the estimators \(\hat{\Gamma}_{k_1}\) and \(\hat{\Gamma}_{k_2}\) for \(1\le k_1 < r < k_2 \le p\). The former has underestimated rank and we will show that it is asymptotically biased, but under some circumstances the bias might be small enough to make it preferable in a bias-variance trade-off. The latter has overestimated rank and we will show that it is consistent, but its variance is inflated when compared to \(\hat{\Gamma}_r\). We first recall the known limiting behaviour of \(\hat{\Gamma}_r\) since this serves as an illustrating case and highlights many of the ideas involved in the study of the other two cases. We then derive the limiting behavior of the estimators with misspecified ranks, which is the first main contribution of this paper.

\subsection{Correctly Specified Rank}

We start with a result due to \cite{anderson2002reduced}. The statement of the Theorem as well as the proof are essentially the same as in \cite{anderson2002reduced}. A proof can be found in Section \ref{sec: proofs}.

\begin{theorem} \label{thm: true_rank}
Define \(\tilde{J}_{12}:= (J_{12} - \Sigma_U^{12}(\Sigma_U^{22})^{-1}J_{22})\) and let \(\textnormal{rank}(\Gamma) = r\). Then,
\begin{equation} \label{eq: asym_true}
    \begin{pmatrix}
        T^{\frac{1}{2}}(\hat{\Gamma}_r^{11} - \Gamma_{11}) \, & \, T(\hat{\Gamma}_r^{12} - \Gamma_{12}) \\
        T^{\frac{1}{2}}(\hat{\Gamma}_r^{21} - \Gamma_{21}) \, & \, T(\hat{\Gamma}_r^{22} - \Gamma_{22})
    \end{pmatrix} \rightarrow_w \begin{pmatrix}
        V_{11}(\Sigma_X^{11})^{-1} \, & \, \tilde{J}_{12} B^{-1} \\
        V_{21}(\Sigma_X^{11})^{-1} \, & \, 0
    \end{pmatrix},
\end{equation}
where $J_{12}, J_{22}$ and $B$ are defined in Lemma \ref{lemma: cross}.
\end{theorem}

Note that the rate of convergence for the right two blocks is \(o_P(T^{-1})\) contrary to the usual reduced rank regression setting of independent observations where the rate of convergence  is \(o_P(T^{-\frac{1}{2}})\) for all blocks. This is because \(T \hat{G}_{21}^{:r}(\hat{G}_{11}^{:r})^T\) and \(T^2 \hat{G}_{21}^{:r}(\hat{G}_{21}^{:r})^T\) are convergent, where \(\hat{G}\) is defined in \eqref{eq: eig1}--\eqref{eq: eig2}, as can be seen in the proof.

 
\subsection{Overestimated Rank} \label{sec: asymp_over}
 
Let the true rank of \(\Pi\) be \(0 < r < p\). We are interested in the reduced rank estimator \(\hat{\Gamma}_{r+m}\) with \(r < r + m \le p\). The above results for \(\hat{\Gamma}_{r}\) suggest that this estimator is consistent and with a limiting behaviour somewhat close to that of \(\hat{\Gamma}_r\) depending on \(m\). To tackle this problem, we first analyze the asymptotics of the last \(n\) columns of \(\hat{G}\). Unfortunately, we cannot directly adopt the methods from the previous section, but in much the same way we start with (\ref{eq: eig1}) and (\ref{eq: eig2}). 

Consider equation (\ref{eq: eig2}) in block matrix notation. Using that \((I_r, T^{\frac{1}{2}}I_n) \hat{G}\) is bounded in probability and that \(T^{\frac{3}{4}}\hat{G}_{21}\) converges in probability to 0 (see proof of Theorem \ref{thm: true_rank}), we find that \(\hat{G}_{21}^T S_{XX}^{21} \hat{G}_{12}\), \(\hat{G}_{11}^T S_{XX}^{12} \hat{G}_{22}\), and \(\hat{G}_{21}^T S_{XX}^{22} \hat{G}_{22}\) are \(o_P(T^{-\frac{1}{4}})\). The top-right block of (\ref{eq: eig2}) then reduces to
\[
\hat{G}_{11}^T S_{XX}^{11} \hat{G}_{12} + o_P(T^{-\frac{1}{4}}) = 0
\]
so that \(\hat{G}_{12} = o_P(T^{-\frac{1}{4}})\). By an analogous argument we get \(\hat{G}_{22}S_{XX}^{22}\hat{G}_{22} = I_n + o_P(T^{-\frac{1}{2}})\) and therefore \(T\hat{G}_{22}\hat{G}_{22}^T = (T^{-1}S_{XX}^{22})^{-1} + o_P(T^{-\frac{1}{2}})\). Note that the least squares estimator is given by
\[
\hat{\Gamma}_{\textnormal{LS}} = S_{\Delta X X}(S_{XX})^{-1} = S_{\Delta X X}\hat{G} \hat{G}^T = \hat{\Gamma}_p
\]
corresponding to the case where \(m = n\). Thus, we have obtained an asymptotic distribution for the least squares estimator, albeit in a slightly indirect way. This will also be a consequence of the following more general result. 

\begin{theorem} \label{thm: over_rank}
Assume that \(\textnormal{rank}(\Gamma)=r< p\) and \(1\le m \le n= p-r\). Then,
\begin{equation} \label{eq: asym_over}
    \begin{pmatrix}
        T^{\frac{1}{2}}(\hat{\Gamma}_{r+m}^{11} - \Gamma_{11}) \, & \, T(\hat{\Gamma}_{r+m}^{12} - \Gamma_{12}) \\
        T^{\frac{1}{2}}(\hat{\Gamma}_{r+m}^{21} - \Gamma_{21}) \, & \, T(\hat{\Gamma}_{r+m}^{22} - \Gamma_{22})
    \end{pmatrix} \rightarrow_w \begin{pmatrix}
        V_{11}(\Sigma_X^{11})^{-1} \, & \, \tilde{J}_{12} B^{-1} + \tilde{J}_{22} P_m\\
        V_{21}(\Sigma_X^{11})^{-1} \, & \, J_{22} P_m
    \end{pmatrix}
\end{equation}
where \(\tilde{J}_{22}:=\Sigma_U^{12}(\Sigma_U^{22})^{-1}J_{22}\) and \(P_m = G_{22}^{:m} (G_{22}^{:m})^T\) with G defined in \eqref{eq: eig_pop}.
\end{theorem}

It is instructive to compare the limiting distribution in \eqref{eq: asym_over} with (\ref{eq: asym_true}). What effectively happens when inflating the rank is that we are including columns of \(\hat{G}\) that are not relevant. This leads to an increased variance as illustrated by the terms \(\tilde{J}_{22}G_{22}^{:m} (G_{22}^{:m})^T\) and \(J_{22}G_{22}^{:m} (G_{22}^{:m})^T\). The higher \(m\) is, the more the variance increases. For small \(m\) compared to \(p\), there might not be any major issues. In line with our intuition, it is thus advisable to get as close as possible to the true rank. Setting \(m=0\) corresponds to dropping all columns of \(G_{22}^{:m}\) and we end up with (\ref{eq: asym_true}). For the least squares estimator  the above expression simplifies somewhat. Indeed, \(G_{22}^{:n} = G_{22}\) and thus \(G_{22}^{:n} (G_{22}^{:n})^T = B^{-1}\). Plugging this into (\ref{eq: asym_over}) yields \(\tilde{J}_{12} B^{-1} + \tilde{J}_{22} G_{22}^{:n} (G_{22}^{:n})^T = J_{12}B^{-1}\) and \(J_{22}G_{22}^{:n} (G_{22}^{:n})^T = J_{22}B^{-1}\).

\subsection{Underestimated Rank} 

For finite samples, we might just as well underestimate the true rank, especially if one chooses the rank using the sequential testing approach that is usually applied in practice. We now consider \(\hat{\Gamma}_{m}\) for \(0 < m < r\). It is clear that the estimator will not be consistent so all we can hope for is that the asymptotic bias is small in certain situations. Before computing this bias and giving the main theorem of this section, we need an extra assumption on the generalized eigenvalues in (\ref{eq: eig_pop}).

\begin{assumption}\label{as: eig}
The first \(r\) generalized eigenvalues in (\ref{eq: eig_pop}) are simple, i.e., \(\lambda_1 >... >\lambda_r\).
\end{assumption}

This assumption is of a technical nature. It is needed for the smoothness results given in Lemma \ref{lemma: eig_vec} and \ref{lemma: eig_diff} in the Appendix. To the extent in which we apply Lemma \ref{lemma: eig_vec}, it is actually sufficient to assume \(\lambda_m > \lambda_{m+1}\). It is clear that this assumption is necessary since otherwise we would not be able to distinguish between the asymptotic eigenvectors. Whether we need all the first \(r\) eigenvalues to be simple, however, is questionable. We hypothesize that Theorem \ref{thm: under_rank} holds without this assumption, but this would require a different proof since the current proof relies on the delta method, which in turn requires sufficient smoothness of a certain map of the generalized eigenvectors. 

It immediately follows from Lemma \ref{lemma: eig_vec} in Appendix \ref{app: aux} and the proof of Theorem \ref{thm: true_rank} that \(\hat{G}_{11}^{:m}(\hat{G}_{11}^{:m})^T\rightarrow_p G_{11}^{:m} (G_{11}^{:m})^T\). Furthermore, we know that \(\hat{\Gamma}^{11}_m = \beta^T\alpha S_{XX}^{11}\hat{G}_{11}^{:m}(\hat{G}_{11}^{:m})^T + o_P(1)\). Then, since \(\beta^T\alpha = \beta^T\alpha\Sigma_X^{11}G_{11}G_{11}^T\), we find that the asymptotic bias is given by
\begin{equation}
\label{eq: b}
\hat{\Gamma}^{11}_m - \Gamma_{11} \rightarrow_p \beta^T\alpha \Sigma_{X}^{11} (G_{11}G_{11}^T - G_{11}^{:m}\left(G_{11}^{:m}\right)^T)=: b_m.
\end{equation}
We see that the asymptotic bias increases as eigenvalues are excluded and the bias is larger for larger eigenvalues. In practice this means that we only incur a small bias when underestimating the rank if the eigenvalues \(\lambda_{m+1},..., \lambda_{r}\) are small.

We obtain the following asymptotic distribution of the reduced rank estimator when the rank is underestimated. A proof can be found in Section \ref{sec: proofs}.

\begin{theorem} \label{thm: under_rank} 
Assume that \(1 \le m < r = \textnormal{rank}(\Gamma)\) and \(\lambda_1 > \cdots >\lambda_{r}\). Let \(\kappa_{ijkl}\) be the joint cumulant of \(U_{t,i}, U_{t,j}, U_{t,k}\), and \(U_{t, l}\) and assume furthermore that \(\kappa_{ijkl}=0\) for all \(1\le i,j,k,l\le p\). Let \(\tilde{V}^T = (\tilde{V}_{11}^T,  \tilde{V}_{21}^T)\) be a random matrix such that \(\textnormal{vec}(\tilde{V})\sim \mathcal{N}(0, \xi\Xi \xi^T)\), where $\Xi$ is defined in \eqref{eq: Xi} and $\xi$ is defined in \eqref{eq: xi}. Then,
\begin{equation} \label{eq: asym_under}
    \begin{pmatrix}
        T^{\frac{1}{2}}(\hat{\Gamma}_{m}^{11} - \Gamma_{11} - b) \, & \, T(\hat{\Gamma}_{m}^{12} - \Gamma_{12}) \\
        T^{\frac{1}{2}}(\hat{\Gamma}_{m}^{21} - \Gamma_{21}) \, & \, T(\hat{\Gamma}_{m}^{22} - \Gamma_{22})
    \end{pmatrix} \rightarrow_w 
    \begin{pmatrix}
        \tilde{V}_{11} \, & \, C_{m}\tilde{J}_{12}B^{-1} \\
        \tilde{V}_{21} \, & \, 0
    \end{pmatrix}
\end{equation}
where \(C_{m}=\beta^T\alpha \Sigma_{X}^{11} G_{11}^{:m}(G_{11}^{:m})^T(\beta^T\alpha)^{-1}\). The covariance matrix of \(\tilde{V}_{21}\) is equal to \(G_{11}^{:m}(G_{11}^{:m})^T\Sigma_X^{11}G_{11}^{:m}(G_{11}^{:m})^T \otimes \Sigma_U^{22}\).
\end{theorem}

From \(G_{11}^{:r}(G_{11}^{:r})^T=(\Sigma_{X}^{11})^{-1}\) it follows that \(C_{r} = I_r\). Comparing \eqref{eq: asym_under} with (\ref{eq: asym_true}) we see that the variances of the top right and bottom left blocks are reduced. The bottom right block also converges in probability to 0. Comparison of the top left blocks is more involved. We could not find a straightforward answer to prefer one over the other. Interestingly enough, simulations suggest that the variance may even increase in certain parts when lowering the rank \(m\). 

When \(m=r\) the expression for \(\xi\) simplifies to the one derived in Theorem \ref{thm: true_rank}, see Section \ref{sec: proofs}.

\subsection{Asymptotics in the Original Coordinates}
\label{sec: orig}

Recall that \(X_t\) defined by \eqref{eq: qvecm} was a transformation of \(Y_t\) defined by \eqref{eq: vecm} into coordinates where the stationary and random-walk parts of the process are separated. Our original parameter of interest was \(\Pi\). We now discuss how to derive central limit theorems for a family of estimators of \(\Pi\) analogous to those discussed above. In particular, we define for \(1\le k\le p\) the matrices \(\hat{L}_k = Q^T \hat{G}_k\) and the estimators \(\hat{\Pi}_k = S_{\Delta Y Y}\hat{L}^{:k}(\hat{L}^{:k})^T\). Then \(\hat{L}^{:k}\) solves 
\[
     S_{Y \Delta Y}(S_{\Delta Y \Delta Y})^{-1}S_{\Delta Y Y} \hat{L}^{:k} = S_{YY}\hat{L}^{:k} \hat{\Lambda}^{:k:k}  
\]
\[
    (\hat{L}^{:k})^T S_{YY} \hat{L}^{:k} = I_k
\]
where \(\Lambda^{:k:k} = \text{diag}(\hat{\lambda}_1, \dots, \hat{\lambda}_k)\) and the solutions to \(|S_{Y \Delta Y}(S_{\Delta Y \Delta Y})^{-1}S_{\Delta Y Y} - \hat{\lambda} S_{YY}|=0\) are the same as those for the \(Q\)-transformed cross-covariances. The columns of \(\hat{L}\) are thus the generalized eigenvectors for the generalized eigenvalue problem given by \(S_{Y \Delta Y}(S_{\Delta Y \Delta Y})^{-1}S_{\Delta Y Y}\) and \(S_{YY}\). Furthermore, we have \(\hat{\Pi}_k = Q^{-1} \hat{\Gamma}_k Q\) and, by definition, \(\Pi = Q^{-1} \Gamma Q\). Consequently, a central limit theorem for \(\hat{\Pi}_k\) is easily obtained from Theorems \ref{thm: true_rank}, \ref{thm: over_rank}, and \ref{thm: under_rank}.

\begin{theorem} \label{thm: asym_pi} 
Assume that \(1 \le r = \textnormal{rank}(\Pi) \le p\). Then, if \(r \le k \le p\),
\[
T^{\frac{1}{2}}\textnormal{vec}(\hat{\Pi}_k - \Pi)\rightarrow \mathcal{N}(0, \beta (\Sigma_X^{11})^{-1}\beta^T\otimes \Sigma_Z).
\]
If we furthermore assume that \(\lambda_1 > ... > \lambda_r > 0\) and that \(\kappa_{ijkl}=0\) for all \(1\le i,j,k,l\le p\), then, for \(1\le k < r\),
\[
T^{\frac{1}{2}}\textnormal{vec}(\hat{\Pi}_k - \Pi - \tilde{b})\rightarrow \mathcal{N}(0, \tilde{\xi}\Xi\tilde{\xi}^T)
\]
where \(\tilde{\xi} = (Q^T\otimes Q^{-1})\xi\) and \(\tilde{b}=\alpha\Sigma_X^{11}(G_{11}G_{11}^T - G_{11}^{:m}(G_{11}^{:m})^T)\beta^T\) is the asymptotic bias.
\end{theorem}

The \(T^{\frac{1}{2}}\) terms dominate in the limiting behaviour of \(\hat{\Pi}_k\), which is why, asymptotically, we lose nothing by overestimating the rank. However, for finite samples the case might be different. As suggested by Theorem \ref{thm: over_rank}, the variance in the random walk direction will increase if we unnecessarily inflate the rank. See \cite{anderson2002reduced} regarding further interpretation of the asymptotics of \(\Pi\).

\section{Estimation Under Rank Uncertainty}
\label{sec: estim}

The above results suggest that the choice of cointegration rank is crucial. While choosing a rank that is too high still results in a consistent estimator, underestimating the rank will result in an asymptotically biased estimator. The rank is usually found using a sequential testing approach as described in \cite{johansen1995likelihood}, which we  briefly recall here. While this approach consistently estimates the true rank (at least if the critical values of the sequential tests go to infinity at an appropriate rate with increasing sample size), disregarding the uncertainty involved in rank estimation from the corresponding post-selection reduced rank estimator might be unfavourable in some cases, especially for high dimension $p$. We therefore suggest a weighted estimator of \(\Pi\), which can be thought of as a weighted average of the estimators \(\hat{\Pi}_1,..., \hat{\Pi}_p\) with either fixed pre-specified weights or with weights inferred from the data. The post-selection estimator obtained by considering the rank-estimate as fixed is a special case where all the weight is assigned to \(\hat{\Pi}_{\hat{r}}\), \(\hat{r}\) being the rank-estimate.

\subsection{Rank Selection}
\label{sec: rank}

We start with the hypothesis \(H(0)\) that the cointegration rank is 0, that is, \(\Pi\) vanishes so that the process is a random walk. This null-hypothesis is tested either against \(H(1)\) or \(H(p)\), which are the hypotheses for cointegration rank 1 and \(p\), respectively. The latter hypothesis corresponds to \(\Pi\) having full rank and thus, under the current assumptions, to a stationary process. If \(H(0)\) is rejected at, say, a 5\% significance level, we move on to the next hypothesis \(H(1)\) and, again, test it either against \(H(2)\) or \(H(p)\). This process is repeated until reaching an \(r\) for which \(H(r)\) cannot be rejected. Assuming that \(Z_0\) is Gaussian, we can directly compute the maximized likelihood function of each hypothesis and thus also a likelihood ratio test statistic. For testing \(H(r)\) against \(H(p)\) the likelihood ratio statistic, \(LR(H(r)|H(p))\), is given by
\[
-2\log LR(H(r)| H(p)) = -T\sum_{i=r+1}^{p}\log (1-\hat{\lambda}_i).
\]
The likelihood ratio statistic for testing \(H(r)\) against \(H(r+1)\) is given by
\[
-2\log LR(H(r)| H(r+1)) = -T\log (1-\hat{\lambda}_{r+1}).
\]
The two test statistics, depending on whether we test against \(H(p)\) or \(H(r+1)\), are usually called the trace and maximum eigenvalue test statistics, respectively. The asymptotics of either can be derived from our discussions above. Assuming that the true rank is \(r_0\), both statistics tend to infinity in probability for \(r<r_0\). The null is therefore rejected when the statistic is larger than a critical value \(c\). In practice, we therefore expect to underestimate the rank in cases where one or more of the population eigenvalues \(\lambda_1, \cdots , \lambda_r\) are close to 0. Recall that the bias is determined only by the smallest \(r-m\) of these eigenvalues where \(m<r\) is the estimated rank and \(r\) the true rank.

Let \(lr=(lr_0, lr_1,..., lr_{p-1})\) denote a sequence of test statistics with either \(lr_k = -2 \log(H(k)|H(p))\) or \(lr_k = -2 \log(H(k)|H(k+1))\) and let \(c_T = (c_{T, 0},..., c_{T, p-1})\in\mathbb{R}_+^p\) be a sequence of critical values. A rank estimate is then given by 
\begin{equation} \label{eq: rank_est}
    \hat{r} = \min\left(\inf\left\{0\le k \le p | lr_k \le  c_{T, k}\right\}, p\right)
\end{equation}
where we use the convention \(\inf \emptyset = \infty\). Usually \(c_{T, k}\) is chosen to be the \((1-\alpha)100 \%\) quantile of the asymptotic distribution of either the trace or maximum eigenvalue test-statistic for some small \(\alpha\in(0, 1)\). Letting \(\alpha\) approach \(0\) with growing sample size at an appropriate rate ensures that \(\hat{r}\) is a consistent estimator of the true rank, \(r_0\).

\subsection{Weighted Reduced Rank Estimator}
\label{sec: weight}

The reduced rank estimators of \(\Gamma\) discussed thus far can all be considered as special cases of a general family of estimators weighting the contribution of each of the eigenvectors in \eqref{eq: eig1}. Indeed, for any \(1\le k\le p\), we can write
\begin{equation} \label{eq: weight_est}
    \hat{\Gamma}_k = S_{\Delta X X} \sum_{i=1}^k \hat{g}_i \hat{g}_i^T = S_{\Delta X X}\sum_{i=1}^d w_i \hat{g}_i \hat{g}_i^T
\end{equation}
where \(w_i = 1\) if \(i \le k\) and \(w_i = 0\) otherwise. For a given vector of weights, \(w\in [0, 1]^p\), with \(w_1\le w_2 \le ... \le w_p\), we refer to the estimator given by \eqref{eq: weight_est} as the \emph{weighted reduced rank estimator} of \(\Gamma\) and write \(\hat{\Gamma}_w\). It is also entirely possible to choose weights that depend on the data. Thus, the post-selection estimator, \(\hat{\Gamma}_{\hat{r}}\), with \(\hat{r}\) as given in \eqref{eq: rank_est}, can be written as a weighted reduced rank estimator with weights 
\begin{equation}
\label{eq: w hat}
\hat{w}_i = \mathbf{1}_{\{i\le \hat{r}\}} = \mathbf{1}_{\{lr_{i-1} > c_{T, i-1}\}}.
\end{equation}
Furthermore, the weighted reduced rank estimator can be viewed as a weighted average of all the individual reduced rank estimators. To see this, define the additional weights \(w_0=0\) and \(w_{p+1}=1\) and let \(W_i=w_{i+1} - w_i\) for \(i=0,..., p\). Then \(W\in[0, 1]^{p+1}\) with \(\sum W_i = 1\) and
\[
\hat{\Gamma}_w = \sum_{k=0}^p W_k\hat{\Gamma}_k
\]
with the convention \(\hat{\Gamma}_0 = 0\).

Assuming that the true rank is \(1\le r_0 \le p\) it is immediately clear from Section \ref{sec: asymp} that any weighting that does not asymptotically assign weight one to all eigenvectors, \(\hat{g}_1,..., \hat{g}_{r_0}\), will result in an asymptotically biased estimator. Conversely, if \(w_i\rightarrow_p 1\) for all \(i=1,...,r_0\), then \(\hat{\Gamma}_w\) is consistent regardless of the asymptotic behaviour of the rest of the weights. Now, for \(w\in[0, 1]^p\), let \(D = \text{diag}(w)\) and write \(D_1 = (D_{i, j})_{i,j\le r}\) and \(D_2 = (D_{i,j})_{i,j > r}\). We define the following quantities:
\begin{align*}
    b_w &= \beta^T\alpha \Sigma_X^{11}G_{11}(I_r - D_1)G_{11}^T, \\
    C_{1w} &= \beta^T\alpha \Sigma_X^{11}G_{11} D_1 G_{11}^T\left(\beta^T\alpha\right)^{-1}, \\
    C_{2w} &= G_{22}D_2G_{22}^T, \\
    \xi_w &= \sum_{i=1}^r w_i\left(\begin{pmatrix} 0_{r\times p} & P_i \end{pmatrix} \otimes \begin{pmatrix} I_p & 0_{p\times r}\end{pmatrix} + (I_r \otimes \Sigma_{\Delta X X}) \xi_i\right)
\end{align*}
where \(\xi_i\) is defined in \eqref{eq: xi_i} and \(P_i = G_{11}e_i (G_{11}e_i)^T\) with \(e_i\) being the \(i\)'th unit vector, that is, \(G_{11}e_i\) is the \(i\)'th column of \(G_{11}\). The following result is a consequence of Theorems \ref{thm: true_rank}, \ref{thm: over_rank} and \ref{thm: under_rank}.

\begin{theorem} \label{thm: asym_weight}
    Assume that \(1 \le r = \textnormal{rank}(\Gamma) \le p\) and \(\lambda_1 > \cdots >\lambda_{r}\). Let \((w_T)_{T\in\mathbb{N}}\subset [0, 1]^p\) be a sequence of weights. Assume that \(w_{T, i}\rightarrow_p 1\) for all \(i=1,..., r\). Then, for \(T\rightarrow\infty\),
    \[
    \hat{\Gamma}_{w_T}\rightarrow_p \Gamma.
    \]
    Assume furthermore that \(T||(w_T - w)||\rightarrow_p 0\) for some \(w\in[0, 1]^p\) and \(\kappa_{ijkl}=0\) for all \(1\le i,j,k,l\le p\). Let \(V_w^T = (\tilde{V}_{w, 11}^T,  \tilde{V}_{w, 21}^T)\) be a random matrix such that \(\textnormal{vec}(V_w)\sim \mathcal{N}(0, \xi_w\Xi \xi_w^T)\), where $\Xi$ is defined in \eqref{eq: Xi} and \(\xi_w\) is defined above. Then,
    \begin{equation} \label{eq: asym_weight}
    \begin{pmatrix}
        T^{\frac{1}{2}}(\hat{\Gamma}_{w_T}^{11} - \Gamma_{11} - b_w) \, & \, T(\hat{\Gamma}_{w_T}^{12} - \Gamma_{12}) \\
        T^{\frac{1}{2}}(\hat{\Gamma}_{w_T}^{21} - \Gamma_{21}) \, & \, T(\hat{\Gamma}_{w_T}^{22} - \Gamma_{22})
    \end{pmatrix} \rightarrow_w 
    \begin{pmatrix}
        V_{w, 11} \, & \, C_{1w}\tilde{J}_{12}B^{-1} + \tilde{J}_{22}C_{2w}\\
        V_{w, 21} \, & \, J_{22}C_{2w}
    \end{pmatrix}.
    \end{equation}
\end{theorem}

It follows from Theorem \ref{thm: asym_weight} that if a sequence of weights, \((w_T)_{T\in\mathbb{N}}\), is such that \(T(w_{T, i} - \mathbf{1}_{\{i\le m\}}) \rightarrow_p 0\) for some \(1\le m\le p\), then \(\hat{\Gamma}_{w_T}\) is asymptotically equivalent to \(\hat{\Gamma}_m\). In particular,  for the post-selection estimator \eqref{eq: rank_est} with weights \eqref{eq: w hat}, for every \(\epsilon > 0\), we have
\begin{align*}
    \mathbb{P}\left(T|\hat{w}_i - \mathbf{1}_{\{i\le r\}}| > \epsilon\right)
        & \le \mathbb{P}\left(\hat{r}\neq r\right) \\
        & = \mathbb{P}\left(\hat{r} > r\right) + \mathbb{P}\left(\hat{r} < r\right) \\
        & \le \sum_{i=0}^{r-1}\mathbb{P}\left(lr_i \le c_{T, i}\right) + \mathbb{P}\left(lr_r > c_{T, r}\right)
\end{align*}
where the last term goes to \(0\) for \(T\rightarrow\infty\) if \(c_{T}\) is chosen appropriately, since \(lr_i\) goes to infinity for \(i<r\) and \(lr_r\) converges in distribution and is therefore bounded in probability. Consequently, the post-selection estimator is asymptotically equivalent to \(\hat{\Pi}_r\).\footnote{This is only true in a pointwise sense. Indeed, the situation is very different if one considers sequences of parameters \(\Pi_T\) with eigenvalues getting arbitrarily close to \(0\). See, for example, the discussion in \cite{elliott1998robustness} or the simulations in the following section. For uniform asymptotic inference in this setting, see \cite{holberg2023uniform}.} In finite samples, however, the situation can be different. It is entirely possible that \(lr_i\) is close to \(c_{T, i}\) for multiple \(i=0,...,r\), which would indicate that there is evidence for multiple ranks in the observed data. Hard threshold weights like \(\hat{w}\) disregard this uncertainty and for some samples the choice of rank can be far from the true rank (see Appendix \ref{app: rank_bias}). It might therefore be wise to explore weights that behave more smoothly. We here give two examples of such weights both of which are based on the likelihood-ratio test statistics, \(lr\). 

\begin{example} \label{ex: w_1}
    The first weight-vector we consider is motivated by the fact that large values of \(lr_i\) are strong evidence that the eigenvector \(\hat{g}_{i+1}\) should be included. Indeed, as stated above, if \(i<r\), then \(T^{-a}lr_i\) goes to infinity for any \(a\in[0, 1)\). It is similar to the weighting scheme considered in \cite{lieb2017inference}. For \(a_1 > 0\) and \(0\le a_2\) we define
    \begin{equation} \label{eq: w_1}
        \hat{w}_1(a_1, a_2) = \left(1 - e^{-a_1 T^{-a_2}lr_0},..., 1 - e^{-a_1 T^{-a_2}lr_{p-1}}\right). 
    \end{equation}
    When convenient we shall omit the arguments and simply write \(\hat{w}_1\). The hyperparameters \(a_1\) and \(a_2\) control how sensitive the weights are to the size of \(\hat{\lambda}\). If \(1 > a_2 > 0\), then \(T^{-a_2}lr_i \rightarrow_p \infty\) for \(i<r\) and \(T^{-a_2}lr_i \rightarrow_p 0\) for \(i\ge r\) which implies that \(T(\hat{w}_1 - \mathbf{1}_{\{\cdot \le r\}})\rightarrow_p 0\) for \(T\rightarrow \infty\), i.e., \(\hat{\Pi}_{\hat{w}_1}\) is asymptotically equivalent to \(\hat{\Pi}_r\).
\end{example}

\begin{example} \label{ex: w_2}
    The second example is a soft threshold version of the categorical \(\hat{w}\). We simply replace the indicator function in \(\hat{w}_i = \mathbf{1}{\{lr_{i-1} > c_{T, i}\}}\) with a sigmoid function. Specifically, let \(\tau:\mathbb{R}\rightarrow [0, 1]\) be a sigmoid function, i.e., monotone and differentiable with \(\tau(0)=0.5\), \(\tau(-\infty) = 0\), and \(\tau(\infty) = 1\), and define for \(a > 0\) and \(c=(c_{0},..., c_{p-1})\) the weights
    \begin{equation} \label{eq: w_2}
        \hat{w}_2(a, c) = \left(\tau\left(a(lr_0 - c_{0})\right),..., \tau\left(a(lr_{p-1} - c_{p-1})\right)\right).
    \end{equation}
    Similar to above, we will sometimes omit the arguments for notational simplicity. In most applications it would make sense to just choose \(c_i\) to be the \((1-\alpha)100\%\) quantile for the asymptotic distribution of the test statistic \(lr_i\) for some prespecified signifcance level \(\alpha\in(0, 1)\) in which case we shall write \(\hat{w}_2(a, \alpha)\). The hyperparameter \(a\) controls the gradient of the sigmoid function with higher values resulting in a sharper separation. One can choose \(a, c\) dependent on \(T\) such that \(a(lr_i - c_{i})\rightarrow_p \infty\) if \(i < r\) and \(a(lr_i - c_ i)\rightarrow_p -\infty\) if \(i \ge r\) in which case \(\hat{\Pi}_{\hat{w}_2}\) is also asymptotically equivalent to \(\hat{\Pi}_r\). This weight works well for moderate dimensions. Indeed, for very high \(p\), choosing the appropriate vector \(c\in\mathbb{R}^p\) becomes prohibitive.
\end{example}

\section{Simulation study}
\label{sec: simul}

In this section we perform two sets of simulation studies. First we compare different weighted reduced rank estimators across a range of parameters. The second set of simulation experiments compares the empirical large-sample distribution of our estimators with the asymptotic distributions derived in Section \ref{sec: asymp}. This will not only confirm our results, but also give an idea of how the distributions in (\ref{eq: asym_true}), (\ref{eq: asym_over}), and (\ref{eq: asym_under}) behave which is useful especially for the last case because of its complicated nature. Details about the different simulation setups are given in Appendix \ref{app: sim}.

\subsection{Comparison of weighted reduced rank estimators}

We compare different weighted reduced rank estimators for a handful of configurations. We compare 4 different types of weights. The first type is given by \(w_f(k)\in[0, 1]^p\) with \((w_f(k))_i = 1\) for \(i\le k\) and 0 otherwise. \(w_f\) does not depend on the observed data but simply chooses a fixed number of eigenvectors to include. It thus corresponds to the simple reduced rank estimators of fixed rank, i.e., \(\hat{\Gamma}_{w_f(k)} = \hat{\Gamma}_k\). The second type is the post-selection weight \(\hat{w}= w_f(\hat{r})\) in  \eqref{eq: w hat}. The last two types are \(\hat{w}_1\) and \(\hat{w}_2\) in eqs. \eqref{eq: w_1} and \eqref{eq: w_2} for different values of hyperparameters. We consider \(w_2(a, \alpha)\) as a function of the first parameter \(a\) and a significance level \(\alpha\in(0, 1)\) as explained in Example \ref{ex: w_2}. We used the error function as the sigmoid function, i.e., 
$$
\tau(x)=\frac{\text{erf}(x) + 1}{2}, \quad \text{erf}(x)=\frac{2}{\sqrt{\pi}}\int_0^x e^{-t^2}dt.
$$

We are interested in settings where there might be uncertainty regarding the choice of rank. This simulation experiment therefore considers sequences of parameters \(\Gamma_c\in\mathbb{R}^{p\times p}\) for which a third of the eigenvalues are stationary (fixed at \(-3/2\)), a third is exactly 0, and the last third is given by \(-c/T\) for varying \(c\ge0\). Throughout we fix \(T=100\). For \(p=3,6,9\) and \(c\in[0, 30]\) we then compare all the estimators based on the mean squared prediction error (MSPE), which, for an estimator \(\hat{\Gamma}\), is given by \(T\mathbb{E}||\Delta X_{T+1} - \hat{\Gamma}X_T||^2\). For a detailed description, see Appendix \ref{app: sim_est}. The results over 4 million simulations are given in Figure \ref{fig: mse_3}. The lines in the figure are smooth versions of the actual results reported in Appendix \ref{app: sim}.

\begin{figure}
    \centering
    \includegraphics[width=1\textwidth]{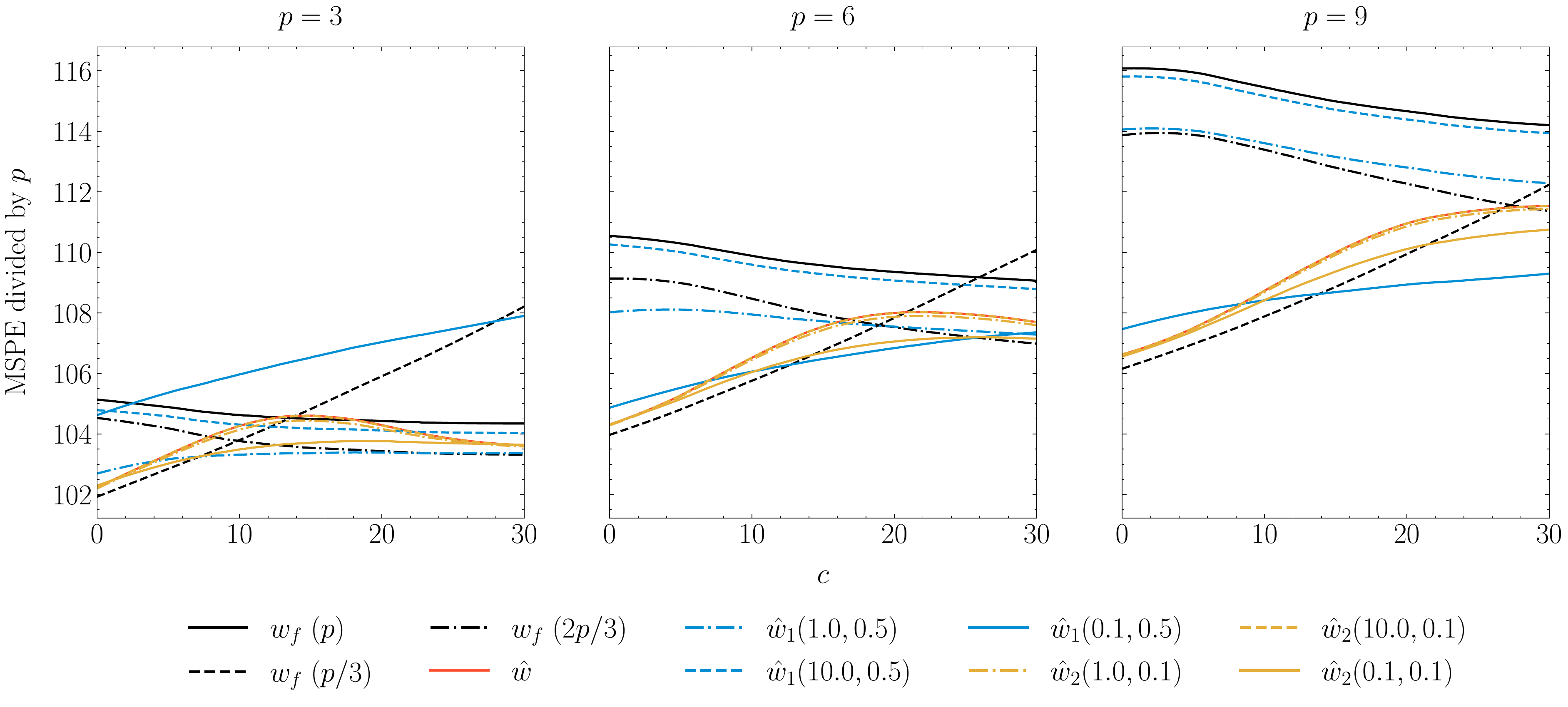}
    \caption{\small Mean square prediction error (MSPE) of different weighted reduced rank estimators for varying dimensions and \(c\in[0, 30]\) where the underlying autoregressive matrix, \(\Gamma_c\), has a third of its eigenvalues set to \(-c/T\), a third set to 0 and a third set to \(-3/2\). Sample size is fixed at \(T=100\). The lines have been smoothed out for better comprehension. See Figure \ref{fig: mse_3_jag} for the true graphs.}
    \label{fig: mse_3}
\end{figure}

For the parameters considered here, the cointegration rank will always be less than two-thirds of the dimension \(p\) and, for \(c\) close to 0, it will be practically \(p/3\). Thus, the least squares estimator is really overparameterized which results in a higher MSPE in almost all cases. This is in line with the asymptotic theory developed in \ref{sec: asymp}. Interestingly, the least squares estimator appears to briefly outperform the post-selection estimator for \(p=3\) and \(c\in[13, 17]\). This can attributed to the additional variance stemming from the rank selection. Also, for \(c\) small enough, the reduced rank estimator of rank \(p/3\) outperforms the reduced rank estimator of rank \(2p/3\) (the dashed and dash-dotted black lines). Thus, choosing a rank smaller than the true rank is beneficial if the discarded eigenvectors have eigenvalues close enough to \(0\). At some point, however, \(c\) is too large and the bias induced by discarding these eigenvector will grow correspondingly at which point the reduced rank estimator of rank \(2p/3\) is preferable (around \(c=14, 19, 27\) from left to right). All the data-dependent weights are attempting to detect this point and act accordingly. For the estimators based on \(\hat{w}_1\) it seems as though the MSPE is shifted depending on the dimension. In higher dimension \(\hat{w}_1(0.1, 0.5)\) is a good choice while \(\hat{w}_1(1, 0.5)\) outperforms the other estimators most of the time for \(d=3\). Similarly, smoothing the rank-selection weights increases the predictive performance of the estimator. Indeed, the weighted reduced rank estimator with weights \(\hat{w}_2(0.1, 0.1)\) clearly outperforms the post-selection estimator in all cases. In Figures \ref{fig: w_mean_3} and \ref{fig: w_std_3} in Appendix \ref{app: sim_est} we have plotted the mean and standard deviation of all the weights across all simulations for \(p=3\) which potentially explains a lot of the differences in performance. Similar behaviour holds for \(p=6\) and \(p=9\).

\subsection{Comparison of asymptotic and empirical distributions}

We consider estimators of \(\Gamma\), comparing each block separately. The dimension is $p=4$, the true rank is chosen to be \(r=2\) and we consider the estimators \(\hat{\Gamma}_1\), \(\hat{\Gamma}_2\), and \(\hat{\Gamma}_4\) corresponding to the three cases of underestimated, correct and overestimated rank. We let \(Z_t\) be i.i.d. normal with \(Z_t\sim\mathcal{N}(0, \Sigma_Z)\). We generate \(\alpha, \beta\in\mathbb{R}^{4\times 2}\) and \(\Sigma_Z\in\mathbb{R}^{4\times 4}\) such that Assumptions \ref{as: roots}, \ref{as: orth}, and \ref{as: eig} are fulfilled. For explicit details on the simulation setup, see Appendix \ref{app: sim}. 

In Fig. \ref{fig: dist} we compare the empirical distributions and the asymptotic distributions of the three estimators. That is, we compare the distributions on the left-hand sides to the right-hand sides of (\ref{eq: asym_true}), (\ref{eq: asym_over}), and (\ref{eq: asym_under}). For the estimators under underestimated rank, we also subtracted the bias in eq. \eqref{eq: b}, which is why it appears to be centered. Observe that the estimators for the true rank (red lines) are not visible in most plots because they overlap with the other lines. This agrees with the theory. For the two leftmost columns, the distribution of \(\hat{\Gamma}_2^{\cdot 1}\) coincides with the distribution of \(\hat{\Gamma}_4^{\cdot 1}\). For the bottom-right block, \(\hat{\Gamma}_2^{22}\) and \(\hat{\Gamma}_1^{22}\) are both singular around 0 which is why the empirical distributions are highly concentrated compared to the empirical distribution of \(\hat{\Gamma}_4\).

\begin{figure}
    \centering
    \includegraphics[width=1\textwidth]{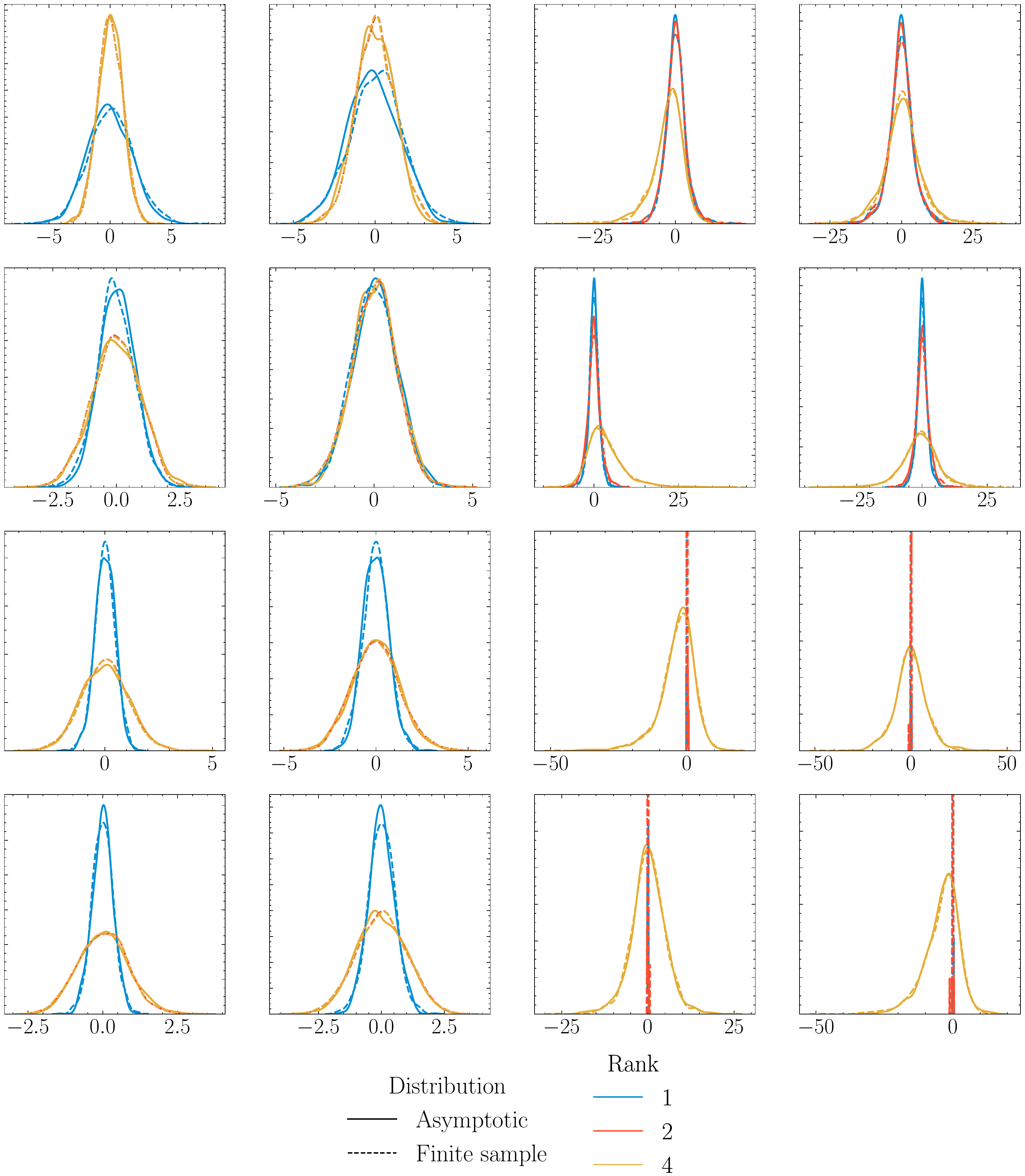}
    \caption{\small Asymptotic and empirical distributions of \(\hat{\Gamma}_k - \Gamma\) for different choices of \(k\). The dimension is $p=4$, the true rank is \(r=2\) and the three estimators are \(\hat{\Gamma}_1\), \(\hat{\Gamma}_2\), and \(\hat{\Gamma}_4\). For each estimator, the dotted line is the empirical distribution for \(T=5000\) and over 1000 simulations. The \(i,j\)'th plot corresponds to the distribution of the \(i,j\)'th element of \(\hat{\Gamma}_k - \Gamma\) . We centered \(\hat{\Gamma}_1-\Gamma\) by subtracting the bias given in the right hand side of \eqref{eq: b} in Section \ref{sec: asymp}. Note different scales in individual plots.}
    \label{fig: dist}
\end{figure}

For all three estimators, the large-sample empirical distribution is close to the asymptotic distribution which confirms our theoretical findings. Furthermore, as we hypothesised, the variance of \(\hat{\Gamma}_1\) is decreased in all blocks except the top-left block. The decrease seems to be most visible in the bottom-left block. This is surprisingly not the case when we compare \(\hat{\Gamma}_1^{11}\) with \(\hat{\Gamma}_2^{11}\). It looks like the former has a higher variance in some of the elements. In other simulations the results were also ambiguous making any quantitative judgements hard. It should be noted, however, that in Fig. \ref{fig: dist} the distribution of each element of the estimated matrix is plotted separately, i.e., we do not consider the covariance structure between different elements of the matrix.

\section{Prediction of EEG Signals}
\label{sec: eeg}

We apply our weighted reduced rank estimators to EEG recordings obtained from an experiment in which two participants were presented with a visual stimulus on a computer screen. Each participant was first shown a cross on the screen on which to focus for a random fixation period between 1.5 and 2.5 seconds. Then, two figures would briefly appear on the screen and the participant should indicate which stimulus had been shown. For more information on the exact setup, see \citet{levakova2022classification}. Here, we analyze the trials from participant 1 which, after data clean-up, amount to 609 trials in total with a sampling rate of 256 Hz. For each trial we only consider the period one second prior to the onset of the visual stimulation. The psychological hypothesis is that the brain state at stimulus onset is predictive of cognitive performance, and this short pre-stimulus period is therefore of special interest. After data preprocessing and clean-up, each EEG signal consists of 59 channels. The resulting data set has 609 observations of 59-dimensional time series of sample size 257. We represent the data by \(X^i_t\in\mathbb{R}^{59}\) where \(i=1,\dots, 609\) and \(t=0,1,\dots, 256\). See Fig. \ref{fig: eeg_sig} for a sample of \(X^i\).

\begin{figure}
    \centering
    \includegraphics[width=.8\textwidth]{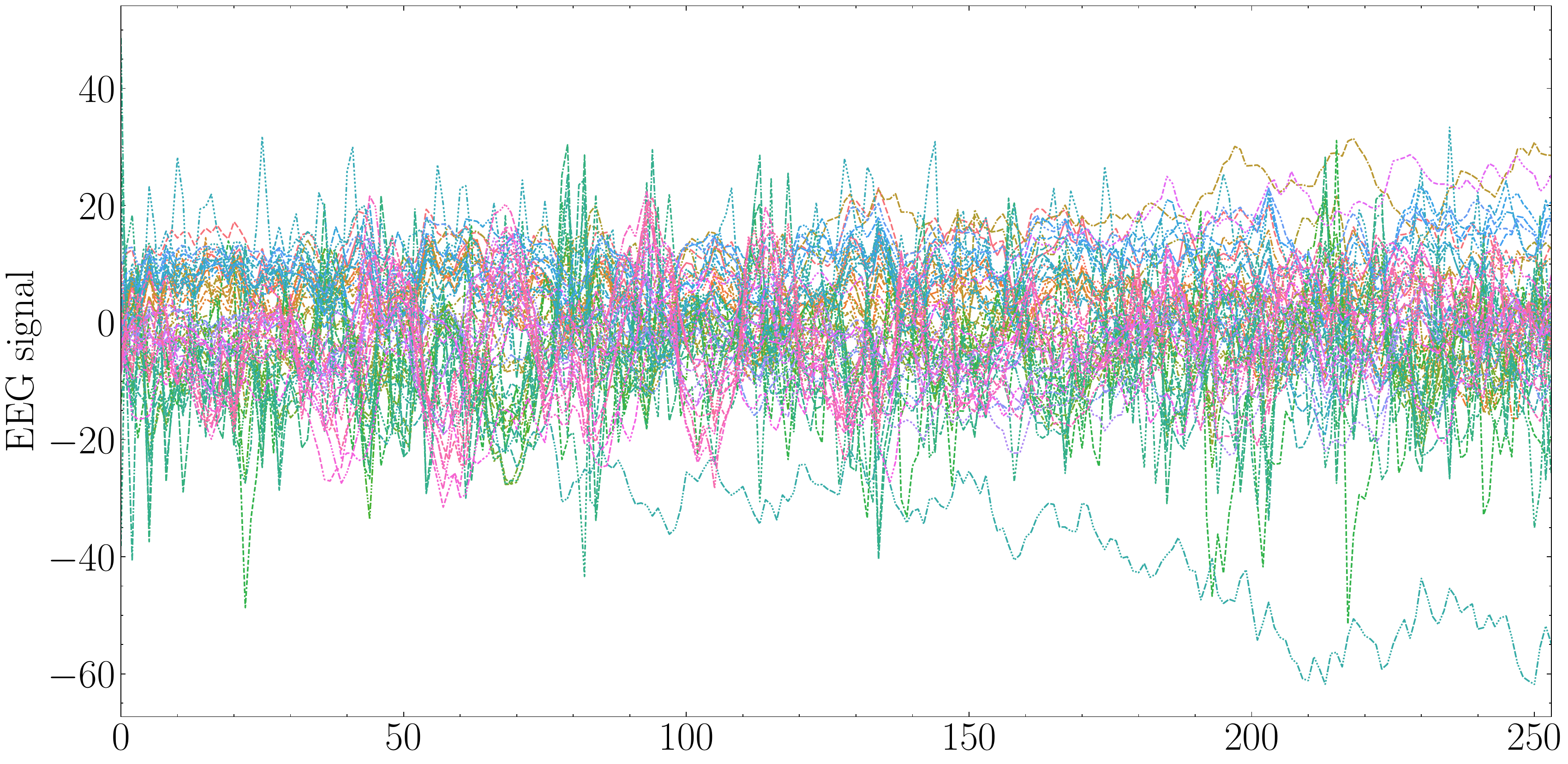}
    \caption{\small Plot of a sample of 59 EEG channels from participant 1 ranging over a second prior to stimulation onset and sampled at 256 Hz.}
    \label{fig: eeg_sig}
\end{figure}

We analyze the predictive capabilities of the weighted reduced rank estimators for two classes of weights on the given data. We consider the discrete weights given by \(w_f(k)\), \(k=0,\dots, 59\) as well as the smooth weights \(\hat{w}_1(1, a_2)\) for \(a_2=k/25\), \(k=0,\dots, 49\). The high dimension of the data makes any classical methods of rank selection as well as methods based on bootstrap prohibitive. Similarly, the weights given in Example \ref{ex: w_2} are not well suited for problems in higher dimension due to the need to select the thresholds \(c\in\mathbb{R}^d\). The methods proposed so far in this paper are in the setting of a single observation of a long time series and under the assumption of zero drift. They are, however, straightforwardly adapted to work in settings with multiple i.i.d. observations of the same time series and to allow for the inclusion of a constant drift \citep[Section 2]{levakova2022classification}.

In Fig. \ref{fig: eeg_w_mspe} we record the performance of each \(w_f(k)\) and \(\hat{w}_1(1, a_2)\) in terms of MSPE. For a test/train split \(I_{train}, I_{test}\subset \{1, \dots, 609\}\) with \(I_{train}\cap I_{test}=\emptyset\), the model was fitted on the train set \((X^i)_{i\in I_{train}}\) and the MSPE calculated on the test set \((X^i)_{i\in I_{test}}\). We compare the estimators for three different sample sizes of the training data. A train size of \(q\) and test size of \(p\) means that \(|I_{train}|=\lfloor q 609\rfloor\) and \(|I_{test}|=\lfloor p 609\rfloor\). Throughout we fix the test size at \(0.1\), i.e., 10\% of the observations are used to compute the MSPE. The results reported in Fig. \ref{fig: eeg_w_mspe} are averaged across 40 random test/train splits of the data for train sizes 0.1, 0.2 and 0.3.

\begin{figure}
    \centering
    \includegraphics[width=0.8\textwidth]{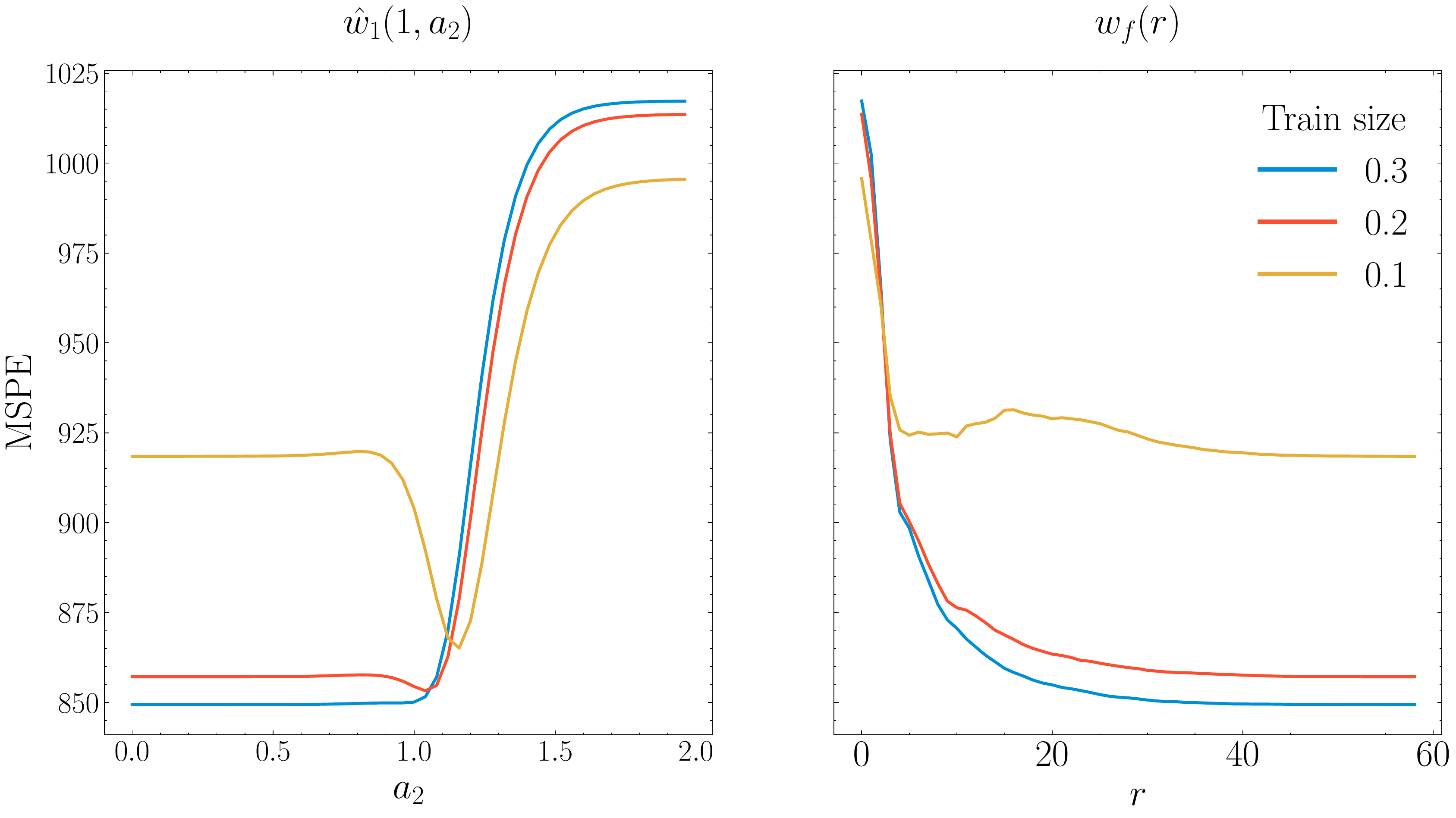}
    \caption{\small Average mean square prediction error (MSPE) for chosen reduced rank estimators over 40 random test/train splits of the data for three different choices of train size. In each case the test size is fixed at 0.1.}
    \label{fig: eeg_w_mspe}
\end{figure}

Evidently, for both choices of weights, the fixed rank and the smoothed weights, the hyperparameter, $r$ or $a_2$, strongly affects the performance of the corresponding estimator. A similar pattern emerges in the left and right panel of Fig. \ref{fig: eeg_w_mspe}. At certain thresholds the predictive capabilities plateau around the same level, namely, for \(a_2\le 0.9\) and \(r\ge 25\). One way to interpret this is that, after a while, increasing the rank of the estimator does not yield better results, i.e., we lose nothing by using a lower rank representation of the underlying dynamics. Similarly, for \(a_2\ge 1.5\), \(\hat{w}_1(1, a_2)\) is practically 0 so that the MSPE in the left panel plateaus at the same level as the MSPE for \(w_f(0)\). Interestingly, whereas the MSPE in the right panel seems to be almost monotone in the hyperparameter, this is not the case in the left panel. Especially for the smallest train size, the MSPE of \(\hat{w}_1\) dips well below the lowest level achieved by \(w_f\). Thus, for small sample sizes, the new estimator with smooth weights performs better. In practice, we do not know the optimal choice of \(a_2\) or \(k\), but this can be partly resolved by cross-validation. Fig. \ref{fig: eeg_bp} depicts the distribution of the MSPE corresponding to the reduced rank estimators with weights \(w_f(\hat{k}_{cv})\) and \(\hat{w}_1(1, \hat{a}_{2, cv})\) where \(\hat{k}_{cv}\) and \(\hat{a}_{2, cv}\) were chosen to yield the lowest MSPE based on cross-validation on the training data with 10 folds.

\begin{figure}
    \centering
    \includegraphics[width=.7\textwidth]{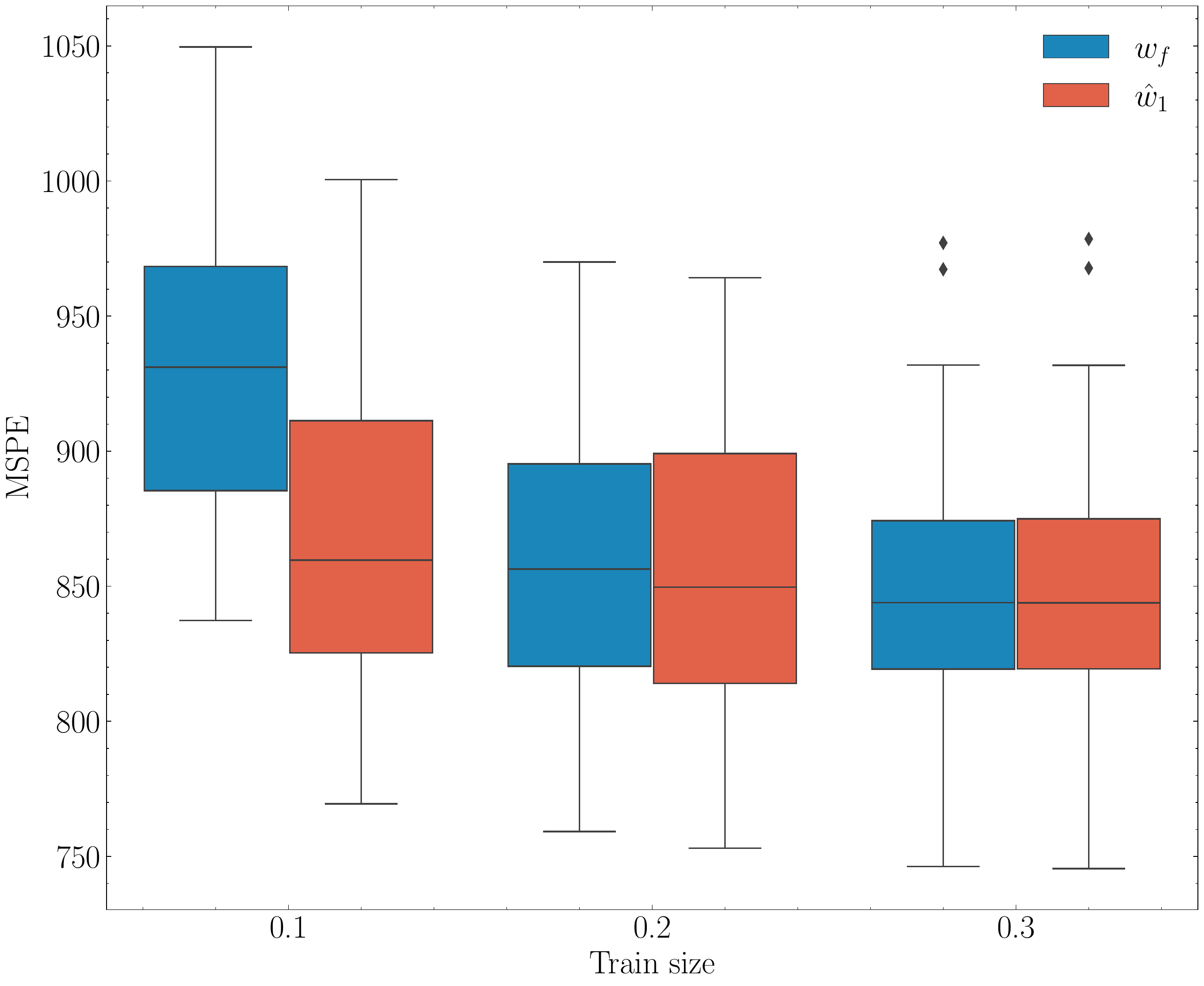}
    \caption{\small Distribution of the mean square prediction error  (MSPE) corresponding to the estimators with weights \(w_f(\hat{k}_{cv})\) and \(\hat{w}_1(1, \hat{a}_{2, cv})\). For each split, the hyperparameters $\hat{k}_{cv}$ and $\hat{a}_{2, cv}$ were chosen by cross-validation on the training data.}
    \label{fig: eeg_bp}
\end{figure}

\begin{figure}
    \centering
    \includegraphics[width=.8\textwidth]{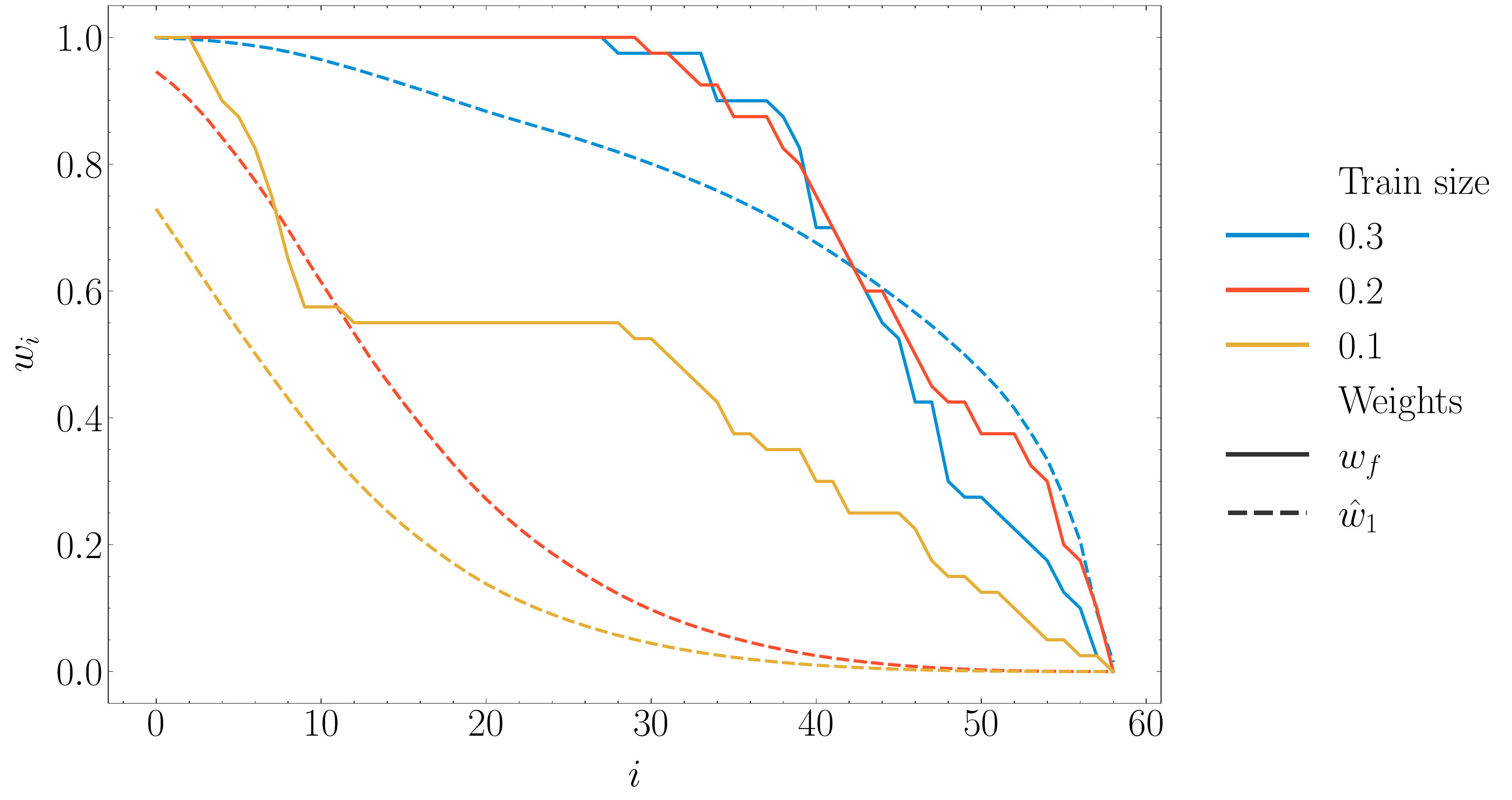}
    \caption{\small Average values of the weights \(w_f(\hat{k}_{cv})_i\) for the classical fixed rank estimator (i.e., weights are 1 for $i \leq \hat{k}_{cv}$ and 0 otherwise) and \(\hat{w}_1(1, \hat{a}_{2, cv})_i\) for the new proposed weighted rank estimator, \(i=1,\dots, 59\), across 40 test/train splits for different choices of train size. For each split, the hyperparameters $\hat{k}_{cv}$ and $\hat{a}_{2, cv}$ were chosen by cross-validation on the training data.}
    \label{fig: eeg_w}
\end{figure}

For train sizes 0.2 and 0.3 the two estimators seem to do equally well. This in line with the results in Fig. \ref{fig: eeg_w_mspe}. However, the situation changes for the smallest train size where the smooth weights clearly outperform the discrete weights. In Fig. \ref{fig: eeg_w} we can see that the weights behave differently. The smooth weights tend rather quickly to 0 for larger ranks as the sample size decreases, but the discrete weights are slower to react. In particular, for the smallest sample the chosen rank varies a lot based on the particular data split. This shows that smoothing the weights is beneficial in settings with large rank uncertainty (in this case because of the high dimension and the small sample size).

\section{Conclusion}
\label{sec: con}

We have characterised the asymptotic distribution of all reduced rank estimators of the \(\Pi\)-matrix in a VECM as given by (\ref{eq: vecm}) assuming the cointegration rank, \(r\), and the dimension, \(p\), are held fixed and the sample size \(T\rightarrow\infty\). Previously, only the asymptotic distribution of the reduced rank estimator with true rank has been studied. We showed that similar results hold if the rank is respectively overestimated or underestimated. In the first case, the estimator is still consistent albeit at the cost of an increased variance. In particular, the bottom-right block of (\ref{eq: asym_over}) no longer converges in probability to zero. This is to be expected since we are effectively including more parameters than necessary. In the second case, the estimator is asymptotically biased and the size of the bias is determined by how much the rank is underestimated. Simulation studies confirmed the theoretical findings.

We have introduced a new class of estimators that outperform the classical estimators in settings where certain eigenvalues close to zero. By choosing appropriate weights that take rank uncertainty into account, the weighted reduced rank estimators have several benefits. They are transparent regarding rank evidence, they have smaller mean square prediction error and the resulting estimators have less variance when compared to the post-selection estimator.

\section{Proofs}
\label{sec: proofs}

\begin{proof}[Proof of Theorem \ref{thm: true_rank}]
For ease of notation we write \(A:=S_{X \Delta X}(S_{\Delta X \Delta X})^{-1}S_{\Delta X X}\). From Lemma \ref{lemma: cross} we find that \(T^{-\frac{1}{2}}A_{12}\), \(T^{-\frac{1}{2}}A_{22}\), \(T^{-\frac{1}{2}}S_{XX}^{21}\), and \(T^{-\frac{1}{2}}S_{XX}^{12}\) are \(o_P(1)\). Since \(S^{11}_{XX}\) and \(T^{-1}S^{22}_{XX}\) are \(O_P(1)\), from (\ref{eq: eig2}) we get that \((\hat{G}_{11}, \hat{G}_{12})\) and \(T^{\frac{1}{2}}(\hat{G}_{21}, \hat{G}_{22})\) are bounded in probability for all \(j=1,...,p\). It follows that \(A_{12}\hat{G}_{21}\), \(A_{22}\hat{G}_{21}\), and \(S^{12}_{XX}\hat{G}_{21}\) are \(o_P(1)\). Finally, note that \(\hat{\Lambda}_{11}\) defined in \eqref{eq: eig1} converges in probability to matrix \(\Lambda_{11}\) and \(\hat{\lambda}_i = O_P(T^{-1})\) for \(i=r+1,..., p\) (see e.g.~\cite{johansen1988statistical}).

Writing (\ref{eq: eig1}) in block matrix notation we then have for \(\hat{G}^{:r}\),
\begin{align}
    A_{11} \hat{G}_{11} + o_P(1) &= S_{XX}^{11} \hat{G}_{11} \hat{\Lambda}_{11} + o_P(1) \label{eq: g_1}\\
    A_{21} \hat{G}_{11} + o_P(1) &= (S_{XX}^{21}\hat{G}_{11} + T^{-1}S_{XX}^{22}T\hat{G}_{21})\hat{\Lambda}_{11}. \label{eq: g_2}
\end{align}
With \(H_1=\left(\frac{1}{T}S_{XX}^{22}\right)^{-1}\left(A_{21}(A_{11})^{-1}S_{XX}^{11} - S_{XX}^{21}\right)\) we compute \(T\hat{G}_{21} = H_1 \hat{G}_{11} + o_P(1)\) which, in particular, implies that \(T\hat{G}_{21}\) is bounded in probability and therefore \(\hat{G}_{21}^TS^{22}_{XX}\hat{G}_{21}\), \(\hat{G}_{21}^TS^{21}_{XX}\hat{G}_{11}\), and \(\hat{G}_{11}^TS^{12}_{XX}\hat{G}_{21}\) are all \(o_P(T^{-\frac{1}{2}})\). Applying (\ref{eq: eig2}) then yields
\[
\hat{G}_{11}^TS_{XX}^{11}\hat{G}_{11} = I_r + o_P(T^{-\frac{1}{2}}),
\]
i.e., \(\hat{G}_{11}\hat{G}_{11}^T = (S_{XX}^{11})^{-1} + o_P(T^{-\frac{1}{2}})\). Then we simply compute the estimator \(S_{\Delta X X}\hat{G}^{:r} (\hat{G}^{:r})^T\) in \eqref{eq: Gammahatk} using the block expressions derived above. This gives us
\begin{align*}
    \hat{\Gamma}_r^{11} &= S_{\Delta X X}^{11}\hat{G}_{11}\hat{G}_{11}^T + o_P(T^{-\frac{1}{2}}) = \beta^T\alpha + S_{UX}^{11}(S_{XX}^{11})^{-1} + o_P(T^{-\frac{1}{2}}) \\
    \hat{\Gamma}_r^{21} &= S_{\Delta X X}^{21}\hat{G}_{11}\hat{G}_{11}^T + o_P(T^{-\frac{1}{2}}) = S_{UX}^{21}(S_{XX}^{11})^{-1} + o_P(T^{-\frac{1}{2}})\\
    \hat{\Gamma}_r^{12} &= T^{-1}S_{\Delta X X}^{11}\hat{G}_{11}\hat{G}_{11}^T H_1^T + o_P(T^{-1}) = T^{-1}\beta^T\alpha H_1^T + o_P(T^{-1}) \\
    \hat{\Gamma}_r^{22} &= o_P(T^{-1})
\end{align*}
Appealing to Lemma \ref{lemma: cross} we see that \(H_1^T\rightarrow_w  (\beta^T\alpha)^{-1}(J_{12} - \Sigma_W^{12}(\Sigma_W^{22})^{-1}J_{22}) B^{-1}\) jointly with (\ref{eq: cross_1}), (\ref{eq: cross_2}), and (\ref{eq: cross_3}). The result of Theorem \ref{thm: true_rank} is then easily derived from the above expressions.
\end{proof}

\begin{proof}[Proof of Theorem \ref{thm: over_rank}]
The main ideas of this proof are similar to those of Theorem \ref{thm: true_rank} and we shall proceed in the same manner. Slightly abusing the notation used so far we let \(\hat{G}_{\cdot 2}=(\hat{G}_{12}^T, \hat{G}_{22}^T)^T\). Equation (\ref{eq: eig1}) translates to 
\begin{equation}
\label{eq: Agm}
A\hat{G}_{\cdot 2} = S_{XX}\hat{G}_{\cdot 2} \hat{\Lambda}_{22}^{:m:m}
\end{equation}
where \(\hat{\Lambda}_{22}^{:m:m}=\textnormal{diag}(\hat{\lambda}_{r+1}, ..., \hat{\lambda}_{r+m})\). Recall that \(\hat{\lambda}_i = O_P(T^{-1})\) for \(i=r+1,..., p\) so that \(\hat{\Lambda}_{22} = O_P(T^{-1})\). Now since \(\hat{G}_{\cdot2}=o_P(T^{-\frac{1}{4}})\) (see the comments made at the start of Section \ref{sec: asymp_over}) and \(S_{XX}^{11}, S_{XX}^{12}=O_p(1)\), it follows that \((S^{11}_{XX}, S^{12}_{XX})\hat{G}_{\cdot 2} \hat{\Lambda}_{22}^{:m:m} = o_P(T^{-1})\). In block matrix notation the top part of eq. \eqref{eq: Agm} simplifies to
\[
A_{11} \hat{G}_{12}^{:m} + A_{12}\hat{G}_{22}^{:m} = o_P(T^{-1})
\]
which, with \(H_2 = -(A_{11})^{-1}A_{12}\), can be rewritten as \(\hat{G}_{12}^{:m} = H_2 \hat{G}_{22}^{:m} + o_P(T^{-1})\). Substituting this expression into the bottom part of eq. \eqref{eq: Agm} and multiplying by \(T^{\frac{1}{2}}\) gives
\[
(A_{22} - A_{21}(A_{11})^{-1}A_{12})T^{\frac{1}{2}}\hat{G}_{22}^{:m} = S_{XX}^{22}T^{\frac{1}{2}}\hat{G}_{22}^{:m}\hat{\Lambda}_{22}^{:m:m} + o_P(T^{-\frac{1}{2}})
\]
By the Davis-Kahan Theorem (see e.g. Theorem 4 in~\cite{yu2015useful}) there exists a random matrix \(L\) solving
\[
(A_{22} - A_{21}(A_{11})^{-1}A_{12})L = T^{-1}S_{XX}^{22}L T\hat{\Lambda}_{22}^{:m:m}, \quad L^T T^{-1}S_{XX}^{22}L = I_n,
\]
and such that \(T^{\frac{1}{2}}\hat{G}_{22}^{:m} = L + o_P(T^{-\frac{1}{2}})\). We shall find the asymptotics of \(L L^T\) and then finish the proof by arguing that \(L\) is sufficiently close to \(T^{\frac{1}{2}}\hat{G}_{22}^{:m}\). Using Lemma \ref{lemma: cross} we compute 
\begin{align*}
    H_2 &\rightarrow_w (\beta^T\alpha \Sigma_X^{11})^{-1}(\Sigma_U^{12} + \Sigma_U^{12}(\Sigma_U^{22})^{-1}J_{22}), \\
    (A_{22} - A_{21}(A_{11})^{-1}A_{12})&\rightarrow_w J_{22}^T(\Sigma_U^{22})^{-1}J_{22}
\end{align*}
jointly with (\ref{eq: cross_1}), (\ref{eq: cross_2}), and (\ref{eq: cross_3}). With probability 1 the generalized eigenvalues on the diagonal of \(\Lambda_{22}\) are all distinct. Lemma \ref{lemma: eig_vec} in Appendix \ref{app: aux} along with the continuous mapping theorem then gives \(L L^T\rightarrow_w G_{22}^{:m} (G_{22}^{:m})^T\) jointly with \(H_2\) and the expressions in Lemma \ref{lemma: cross}. We now have all the tools needed to evaluate \(\hat{\Gamma}_{r+m}\) starting with the expression
\[
\hat{\Gamma}_{r+m} = S_{\Delta X X}\begin{pmatrix}\hat{G}_{\cdot 1} & \hat{G}_{\cdot 2}\end{pmatrix}\begin{pmatrix}\hat{G}_{\cdot 1}^T \\ \hat{G}_{\cdot 2}^T\end{pmatrix} = S_{\Delta X X}(\hat{G}_{\cdot 1} \hat{G}_{\cdot 1}^T + \hat{G}_{\cdot 2} \hat{G}_{\cdot 2}^T).
\]
As mentioned above \(\hat{G}_{\cdot2}=o_P(T^{-\frac{1}{4}})\) which implies that \(\hat{G}_{12}^{:m}(\hat{G}_{12}^{:m})^T\) and \(\hat{G}_{22}^{:m}(\hat{G}_{12}^{:m})^T\) are \(o_P(T^{-\frac{1}{2}})\) and so we immediately get \(\hat{\Gamma}_{r+m}^{\cdot 1} = \hat{\Gamma}_{r}^{\cdot 1} + o_P(T^{-\frac{1}{2}})\). For the remaining two blocks write
\begin{align*}
    \begin{pmatrix}
    \hat{\Gamma}_{r+m}^{12} \\
    \hat{\Gamma}_{r+m}^{22} 
\end{pmatrix} - \begin{pmatrix}
    \hat{\Gamma}_{r}^{12} \\
    \hat{\Gamma}_{r}^{22} 
\end{pmatrix}&=  \begin{pmatrix}
    S_{\Delta X X}^{11}\hat{G}_{12}^{:m}(\hat{G}_{22}^{:m})^T + S_{\Delta X X}^{12}\hat{G}_{22}^{:m}(\hat{G}_{22}^{:m})^T \\
    S_{\Delta X X}^{21}\hat{G}_{12}^{:m}(\hat{G}_{22}^{:m})^T + S_{\Delta X X}^{22}\hat{G}_{22}^{:m}(\hat{G}_{22}^{:m})^T
\end{pmatrix} \\
        &= \begin{pmatrix}
            T^{-1}\beta^T\alpha S_{X X}^{11} H_2 LL^T + T^{-1}S_{\Delta X X}^{12}LL^T \\
            T^{-1}S_{\Delta X X}^{22}LL^T
        \end{pmatrix} + o_P(T^{-1})
\end{align*}
and the result follows from Theorem \ref{thm: true_rank} and Lemma \ref{lemma: cross} in combination with the limits derived above for \(H_2\) and \(LL^T\). 
\end{proof} 

Note that the reasoning used to determine the limit of \(LL^T\) can also be applied to \(\hat{\Lambda}_{22}^{:m:m}\). In particular, Lemma \ref{lemma: eig_val} in the Appendix shows that \(T\hat{\Lambda}_{22}^{:m:m}\) converges in distribution to \(\Lambda_{22}^{:m:m}\). There is nothing special about our choice of \(m\) here and in particular \(T\hat{\Lambda}_{22}\rightarrow_w \Lambda_{22}\). It is seen that 
\begin{multline*}
    |J_{22}^T(\Sigma_U^{22})^{-1}J_{22} - B\lambda| = \\|(\Sigma_U^{22})^{\frac{1}{2}}| |(\Sigma_U^{22})^{-\frac{1}{2}}J_{22}^T(\Sigma_U^{22})^{-1}J_{22}(\Sigma_U^{22})^{-\frac{1}{2}} - (\Sigma_U^{22})^{-\frac{1}{2}}B(\Sigma_U^{22})^{-\frac{1}{2}}\lambda| |(\Sigma_U^{22})^{\frac{1}{2}}|.
\end{multline*}
Recalling the definition of \(J_{22}\) and \(B\) in Lemma \ref{lemma: cross} we get that the diagonal of \(\Lambda_{22}\) is, in fact, equal to the ordered solutions of
\[
\left|\left(\int_0^1 W_{2s} \, d W_{2s}^T\right)\left(\int_0^1 W_{2s} \, d W_{2s}^T\right)^T - \lambda\int_0^1 W_{2s}W_{2s}^T \, ds\right|=0
\]
where \(W_{2s}\) are the last \(n\) components of the standard Brownian motion \(W_s\). Analogously, we see from Lemma \ref{lemma: eig_val} and the proof of Theorem \ref{thm: true_rank} that \((\hat{\lambda}_1, ..., \hat{\lambda}_r)\) are asymptotically equivalent to the ordered solutions of \(|A_{11} - S_{XX}^{11}r|\). In other words, \(\hat{\Lambda}_{11}\) converges in probability to \(\Lambda_{11}\) whose diagonal are the ordered solutions of
\begin{equation}\label{eq: population_eig}
  |\Sigma_X^{11}\alpha^T\beta(\Sigma_{\Delta X}^{-1})_{11}\beta^T\alpha \Sigma_X^{11} - \Sigma_X^{11} \lambda| = 0.  
\end{equation}

We have thus determined the asymptotics of \(\hat{\Lambda}\) as well. This is a well known result in the cointegration literature from which one can derive the asymptotic distribution of the so-called trace test statistic, which tests the hypothesis that the cointegration rank is at most \(k<p\)~\citep{johansen1988statistical}.

The asymptotics are a little more involved in the case where the true rank is underestimated. Before the proof, we first show some intermediate lemmas. To study the limiting distribution of \(T^{\frac{1}{2}}(\hat{\Gamma}_m^{11} - \Gamma_{11} - b_m)\) we follow the strategy of ~\cite{izenman1975reduced} which forces us to set up more notation and introduce some ideas from matrix differential calculus. We use the notation from \cite{magnus2019matrix}. For a matrix valued function \(\Phi:\mathbb{R}^{m\times n} \rightarrow \mathbb{R}^{k\times l} \) we let \(d\Phi\) denote its differential. Similarly, we define the derivative of \(\Phi(A)\) with respect to \(A\) as the derivative of the vectorization of \(\Phi(A)\) with respect to the vectorization of \(A\):
\[
D\Phi = \frac{\partial \text{vec}\Phi(A)}{\partial \text{vec}A^T},
\]
i.e., the Jacobian matrix of \(\text{vec}(\Phi)\). One useful result we shall use is the following \citep{neudecker1968kronecker}: If \(d\Phi(A) = \sum_{i} M_i (dA) N_i\) for suitable matrices \(M_i, N_i\), then the derivative is \(D\Phi = \sum_i N_i^T\otimes M_i\). We define the commutation matrix \(I_{(k,l)}\) as the square \(kl \times kl\) block matrix partitioned into \(k\times l\) blocks whose \((i,j)\)'th block is 1 in the \((j, i)\)'th coordinate and 0 everywhere else. Our goal is to use the delta method to determine the asymptotic distribution of the left side of \(\hat{\Gamma}_m\).

\begin{lemma}[Delta method] \label{lemma: delta}
Let \((x_n)_{n\in\mathbb{N}}\subset \mathbb{R}^d\) be a sequence of random vectors such that \(\sqrt{n}(x_n - x) \rightarrow_w \mathcal{N}(0, \Sigma)\) for some \(x\in\mathbb{R}^d\) and a positive definite covariance matrix \(\Sigma \in \mathbb{R}^{d\times d}\). Assume furthermore that \(h:\mathbb{R}^d\rightarrow \mathbb{R}^p\) is a continuous function that is continuously differentiable in a neighbourhood of \(x\) with Jacobian matrix \(J=\frac{\partial h}{\partial y^T}\rvert_{y=x}\). Then, \(\sqrt{n}(h(x_n) - h(x))\rightarrow_w\mathcal{N}(0, J\Sigma J^T)\).
\end{lemma}
Furthermore, we need the following results on the sample covariance matrix.
\begin{lemma}\label{lemma: xtil}
Define \(\tilde{X}_t = \begin{pmatrix} \Delta X_t^T &  X_{1t-1}^T\end{pmatrix}^T\) and consider the sample covariance matrix \(S_{\tilde{X}\tilde{X}}\). Then \(\tilde{X}_t\) is stationary with,
\[
S_{\tilde{X}\tilde{X}}\rightarrow_p \Sigma_{\tilde{X}} = \begin{pmatrix}
    \Sigma_{\Delta X}^{11} \, & \, \Sigma_{\Delta X}^{12} \, & \, \beta^T\alpha \Sigma_{X}^{11} \\
    \Sigma_{\Delta X}^{21} \, & \, \Sigma_{\Delta X}^{22} \, & \, 0 \\
    \Sigma_{X}^{11}\alpha^T\!\beta \, & \, 0 \, & \, \Sigma_X^{11}
\end{pmatrix}.
\]
Let \(\kappa_{ijkl}\) be the joint cumulant of \(U_{t,i}, U_{t,j}, U_{t,k}\), and \(U_{t, l}\). If \(\kappa_{ijkl}\) vanishes for all \(1\le i,j,k,l\le p\), then \(\sqrt{T}\textnormal{vec}(S_{\tilde{X}\tilde{X}} - \Sigma_{\tilde{X}})\rightarrow_w \mathcal{N}(0, \Xi)\), where
\begin{equation}
\label{eq: Xi}
\Xi = \sum_{k=-\infty}^{\infty} \gamma_k \otimes \gamma_k + I_{(p+r, p+r)}\sum_{k=-\infty}^{\infty} \gamma_k\otimes\gamma_k
\end{equation}
and \(\gamma_k = \mathbb{E}(\tilde{X}_0\tilde{X}_k^T)\).

\end{lemma}

Note that the second part of Lemma \ref{lemma: xtil} holds also if we replace \(\tilde{X}_t\) with any multivariate stationary linear process with vanishing fourth order cumulants. In fact, it is sufficient to assume that the cumulants are finite, but this complicates the expression for the asymptotic covariance somewhat so we keep the assumption. It holds specifically when \(U_t\) is Gaussian.

\begin{proof}[Proof of Lemma \ref{lemma: xtil}]
\(\tilde{X}_t\) is clearly stationary since \(\Delta X_t\) and \(X_{1t-1}\) are stationary. From (\ref{eq: xone}) and (\ref{eq: xtwo}) we see that \(\tilde{X}_t = \sum_{s=0}^{\infty}\Psi_s U_{t - s}\) where 
\[
\Psi_0=\begin{pmatrix}
    I_r \, & \, 0_{r\times n} \\
    0_{n\times r} \, & \, I_n \\
    0_{r\times r} \, & \, 0_{r\times n}
\end{pmatrix}, \quad\quad \Psi_s = \begin{pmatrix}
    \beta^T\alpha(I_r + \beta^T\alpha)^{s-1} \, & \, 0_{r\times n} \\
    0_{n\times r} \, & \, 0_{n\times n} \\
    (I_r + \beta^T\alpha)^{s-1} \, & \, 0_{r\times n}
\end{pmatrix} \text{ for } s \ge 1
\]
and it is easily verified that \(\sum_{s=0}^{\infty}||\Psi_s|| < \infty\). The first part of the statement then follows from known results about linear processes (see e.g. Proposition C.12 in \cite{lutkepohl2005new}).

Under the assumptions of Lemma \ref{lemma: xtil}, \(\sqrt{T}(S_{\tilde{X}\tilde{X}} - \Sigma_{\tilde{X}})\) converges in distribution to some random matrix, \(N\), with \(\textnormal{vec}(N)\) normal and covariance given by (see e.g. \cite{roy1989asymptotic})
\[
\textnormal{Cov}(N_{i,j}, N_{k,l}) = \sum_{u=-\infty}^{\infty}(\gamma_u)_{ik}(\gamma_u)_{jl} + \sum_{u=-\infty}^{\infty}(\gamma_u)_{jk}(\gamma_u)_{il}.
\]
Now let \(\eta(i,j)=(p+r)(j-1) + i\) and observe that 
\begin{align*}
    (\gamma_u)_{ik}(\gamma_u)_{jl} &= (\gamma_u \otimes \gamma_u)_{\eta(i, j), \eta(k, l)}, \\
    (\gamma_u)_{jk}(\gamma_u)_{il} &= \left((\gamma_u \otimes \gamma_u) I_{(p+r, p+r)}\right)_{\eta(i, j), \eta(k, l)}.
\end{align*}
Since \(\Xi_{\eta(i, j), \eta(k, l)} = \textnormal{Cov}(N_{i,j}, N_{k,l})\), this is exactly what we need to show, keeping in mind that \(I_{(p+r, p+r)}(\gamma_u \otimes \gamma_u) = (\gamma_u \otimes \gamma_u)I_{(p+r, p+r)}\) \citep{magnus1979commutation}.
\end{proof}

We now have all the tools we need to derive the asymptotics of \(\hat{\Gamma}_m^{\cdot 1}\). For a matrix \(M\in \mathbb{R}^{(p+r)\times(p+r)}\) write it in block form
\[
M = \begin{pmatrix}
    M_{11} & M_{12} \\
    M_{21} & M_{22}
\end{pmatrix}
\]
where \(M_{11}\) is \(p\times p\) and \(M_{22}\) is \(r\times r\). Denote by \(\rho_1 \ge \cdots \ge \rho_r\) the generalized eigenvalues sorted in decreasing order for the generalized eigenvalue problem given by  \(M_{21}(M_{11})^{-1}M_{12}\) and \(M_{22}\). Let \(v_1, ..., v_r\) be the corresponding generalized eigenvectors. Define the function \(h:\mathbb{R}^{(p+r)\times(p+r)}\rightarrow \mathbb{R}^{r\times r}\) by \(h(M) = M_{12}\sum_{k=1}^m v_k v_k^T\). We can write
\[
dM_{11} = \left(I_p, 0_{p\times r}\right) dM \begin{pmatrix}I_p \\ 0_{r\times p}\end{pmatrix}, \quad dM_{12} = \left(I_p, 0_{p\times r}\right) dM \begin{pmatrix}0_{p\times r} \\ I_r\end{pmatrix},
\]
\[
dM_{21} = \left(0_{r\times p}, I_r\right) dM \begin{pmatrix}I_p \\ 0_{r\times p}\end{pmatrix}, \quad dM_{22} = \left(0_{r\times p},  I_r\right) dM \begin{pmatrix}0_{p\times r} \\ I_r\end{pmatrix}.
\]
Also, we have \(dM_{11}^{-1} = -M_{11}^{-1}(dM_{11})M_{11}^{-1}\) (see e.g. Thm. 8.3 in \cite{magnus2019matrix}) whence
\begin{multline*}
    d(M_{21}M_{11}^{-1}M_{12}) 
    = \left(0_{r\times p}, I_r\right) dM \begin{pmatrix} M_{11}^{-1}M_{12} \\ 0_{r\times p}\end{pmatrix} 
    \\- \left(M_{21}M_{11}^{-1}, 0_{p\times r}\right) dM  \begin{pmatrix}M_{11}^{-1}M_{12} \\ 0_{r\times p}\end{pmatrix} 
    + \left(M_{21}M_{11}^{-1}, 0_{p\times r}\right) dM \begin{pmatrix}0_{p\times r} \\ I_r\end{pmatrix}.
\end{multline*}

For ease of notation we now write \(P_i = G_{11}e_i (G_{11}e_i)^T\) with \(e_i\) being the \(i\)'th unit vector, that is, \(G_{11}e_i\) is the \(i\)'th column of \(G_{11}\), \(\Sigma_{X\Delta X} = (\Sigma_X^{11}\alpha^T\beta, 0_{r\times n})\), and \(\Sigma_{\Delta X X} = \Sigma_{X\Delta X}^T\). Under Assumption \ref{as: eig}, the map \(M\mapsto v_i v_i^T\) is differentiable at \(M=\Sigma_{\tilde{X}}\) (see Appendix \ref{app: aux}). Let \(\xi_i = D(v_iv_i^T)\rvert_{M=\Sigma_{\tilde{X}}}\). Lemma \ref{lemma: eig_diff} yields
\begin{equation} \label{eq: xi_i}
    \xi_i = \sum_{j\neq i}(\lambda_i - \lambda_j)^{-1}(P_i\otimes P_j + P_j\otimes P_i)F_i - \left(0_{r\times p}, P_i\right)\otimes \left( 0_{r\times p}, P_i\right)
\end{equation}
where
\[
F_i = \left( \Sigma_{X\Delta X}\Sigma_{\Delta X}^{-1}\otimes \left( -\Sigma_{X\Delta X}\Sigma_{\Delta X}^{-1}, I_r\right), I_r \otimes \left( \Sigma_{X\Delta X}\Sigma_{\Delta X}^{-1}, -\lambda_i I_r\right)\right).
\]
Then, 
\begin{equation}
\label{eq: xi}
\xi = Dh\rvert_{M=\Sigma_{\tilde{X}}} = \sum_{i=1}^m\begin{pmatrix} 0_{r\times p} & P_i \end{pmatrix} \otimes \begin{pmatrix} I_p & 0_{p\times r}\end{pmatrix} + (I_r \otimes \Sigma_{\Delta X X}) \xi_i.
\end{equation}
Observe that (\ref{eq: g_1}) also holds with \(o_P(1)\) replaced by \(o_P(T^{-\frac{1}{4}})\) and a similar argument as that applied in the proof of Theorem \ref{thm: over_rank} therefore shows that \(\hat{G}_{11}^{:m}(\hat{G}_{11}^{:m})^T = G_{11}^{:m}(G_{11}^{:m})^T + o_P(T^{-\frac{1}{2}})\). In particular, 
\[
\sqrt{T}\begin{pmatrix}
    \hat{\Gamma}_m^{11} - \Gamma_{11} - b \\
    \hat{\Gamma}_m^{21} - \Gamma_{21}
\end{pmatrix} = \sqrt{T}(h(S_{\tilde{X}\tilde{X}}) - h(\Sigma_{\tilde{X}})) + o_P(T^{-\frac{1}{2}})
\]
and it is a straightforward application of Lemma \ref{lemma: delta} and Lemma \ref{lemma: xtil} to prove that \(\sqrt{T}\textnormal{vec}(h(S_{\tilde{X}\tilde{X}}) - h(\Sigma_{\tilde{X}})) \rightarrow_w \mathcal{N}(0, \xi \Xi \xi^T)\). Thus, we have identified the asymptotic distribution of the two left blocks. As we shall see below there is a much simpler expression for the asymptotic covariance matrix of \(\sqrt{T}\textnormal{vec}(\hat{\Gamma}_m^{21} - \Gamma_{21})\). We are now ready to prove Theorem \ref{thm: under_rank}.

\begin{proof}[Proof of Theorem \ref{thm: under_rank}]
Starting as in the proof of Theorem \ref{thm: true_rank} and replacing \(\hat{G}_{11}^{:m}(\hat{G}_{11}^{:m})^T\) with \(G_{11}^{:m}(G_{11}^{:m})^T\) in the appropriate places, we find that
\begin{align*}
    \hat{\Gamma}_m^{11} &= S_{\Delta X X}^{11}G_{11}^{:m}(G_{11}^{:m})^T + o_P(T^{-\frac{1}{2}}) \\
    \hat{\Gamma}_m^{21} &= S_{UX}^{21}G_{11}^{:m}(G_{11}^{:m})^T + o_P(T^{-\frac{1}{2}})\\
    \hat{\Gamma}_m^{12} &=  T^{-1}\beta^T\alpha \Sigma_{XX}G_{11}^{:m}(G_{11}^{:m})^T H_1^T + o_P(T^{-1}) \\
    \hat{\Gamma}_m^{22} &= o_P(T^{-1})
\end{align*}
We derived the asymptotic distribution for the first two expressions above. The other two follow directly from Lemma \ref{lemma: cross}.

The second result is also an easy consequence of Lemma \ref{lemma: cross}, since \(T^{\frac{1}{2}}\textnormal{vec}(\hat{\Gamma}_m^{21} - \Gamma_{21})\) converges in distribution to \((G_{11}^{:m}(G_{11}^{:m})^T\otimes I_n)\textnormal{vec}(V_{21})\), which, of course, is normal with mean 0 and covariance matrix as given in the theorem. Another way to arrive at the same result is to first observe that 
\[
\textnormal{vec}\left(\hat{\Gamma}_m^{21} - \Gamma_{21}\right) = \left(I_r \otimes \left( 0_{n\times r}, I_n \right)\right) \textnormal{vec}\begin{pmatrix} \hat{\Gamma}_m^{11} - \Gamma_{11} - b \\ \hat{\Gamma}_m^{21} - \Gamma_{21}\end{pmatrix}
\]
and the asymptotic covariance of \(\sqrt{T}\textnormal{vec}(\hat{\Gamma}_m^{21} - \Gamma_{21})\) must therefore equal \(\xi_{21}\Xi\xi_{21}^T\) where \(\xi_{21} = (I_r \otimes (0_{n\times r}, I_n))\xi\). We compute
\[
\xi_{21} = \left( 0_{r\times p}, G_{11}^{:m}(G_{11}^{:m})^T \right) \otimes \left( 0_{n\times r}, I_n, 0_{n\times r} \right)
\]
and thus \(\xi_{21}(\gamma_k \otimes \gamma_k) = (\gamma_k \otimes \gamma_k)\xi_{21}^T=0\) for all \(k\neq 0\). Furthermore, we find that \(\xi_{21}I_{(p+r, p+r)}(\gamma_0 \otimes \gamma_0)\xi_{21}^T = 0\) and 
\[
\xi_{21}(\gamma_0 \otimes \gamma_0)\xi_{21}^T = \xi_{21}(\Sigma_{\tilde{X}} \otimes \Sigma_{\tilde{X}})\xi_{21}^T = G_{11}^{:m}(G_{11}^{:m})^T\Sigma_X^{11}G_{11}^{:m}(G_{11}^{:m})^T \otimes \Sigma_U^{22}
\]
which then results in the same covariance as before.
\end{proof}

When \(m=r\) the expression for \(\xi\) simplifies significantly. Indeed, as noted after Lemma \ref{lemma: eig_diff} in Appendix \ref{app: aux}, we find that \(\sum_{i=1}^r\xi_i = -(\Sigma_X^{11})^{-1} \otimes (\Sigma_X^{11})^{-1}\) and thus
\[
\xi = \begin{pmatrix} 0_{r\times p} & (\Sigma_X^{11})^{-1}\end{pmatrix} \otimes \begin{pmatrix} I_p & - \Sigma_{\Delta X X}(\Sigma_X^{11})^{-1} \end{pmatrix}.
\]
We then compute \(\xi(\gamma_k\otimes \gamma_k)\xi^T = \xi I_{(p+r, p+r)}(\gamma_k\otimes \gamma_k)\xi^T = 0\) for \(k\neq 0\). For \(k=0\) we have \(\gamma_0 = \Sigma_{\tilde{X}}\) and \(\xi I_{(p+r, p+r)}(\gamma_0\otimes\gamma_0)\xi^T = 0\).  Thus,
\[
\xi \Xi\xi^T = \xi(\Sigma_{\tilde{X}}\otimes\Sigma_{\tilde{X}})\xi^T = (\Sigma_X^{11})^{-1}\otimes \Sigma_U
\]
which is the covariance matrix of \(\textnormal{vec}(V(\Sigma_X^{11})^{-1})\), i.e., our result is in line with the one derived in Theorem \ref{thm: true_rank}.

The following proof is an easy consequence of the discussion in Section \ref{sec: orig}.

\begin{proof} [Proof of Theorem \ref{thm: asym_pi}]
Assume that \(k\ge r\). The first statement is then a direct consequence of  (\ref{eq: asym_true}), (\ref{eq: asym_over}) and (\ref{eq: asym_under}), and the fact that \(T^{\frac{1}{2}}\textnormal{vec}(\hat{\Pi}_k - \Pi) = (Q^T\otimes Q^{-1})T^{\frac{1}{2}}\textnormal{vec}(\hat{\Gamma}_k - \Gamma)\), recalling that the two right blocks are \(o_P(T^{-\frac{1}{2}})\). 

For the second part, assume that \(1 \le k < r\). Then,
\[
\hat{\Pi}_k - \Pi = Q^{-1}(\hat{\Gamma} - \Gamma)Q \rightarrow_p \alpha(\beta^T\alpha)^{-1} b \beta^T = \tilde{b}.
\]
Applying the same argument as above and referring to the proof of Theorem \ref{thm: over_rank}, we obtain the desired distribution.
\end{proof}

\begin{proof}[Proof of Theorem \ref{thm: asym_weight}]
    To prove consistency simply observe that
    \[
    \hat{\Gamma}_{w_T} - \hat{\Gamma}_r = S_{\Delta X X}\left(\sum_{i=1}^r (w_{T, i} - 1)\hat{g}_i \hat{g}_i^T + \sum_{i=r+1}^p w_{T, i}\hat{g}_i \hat{g}_i^T\right).
    \]
    \(S_{\Delta X X} \hat{g}_i\hat{g}_i^T\) is bounded in probability for \(i \le r\) and converges in probability to 0 for \(i > r\) from which it follows that both terms on the right hand side converge in probability to 0 for \(T\) going to infinity. The result then follows since \(\hat{\Gamma}_r\) is consistent.

    Now, for \(0\le k\le p\), define the matrices
    \[
    D = \begin{pmatrix}
        T^{\frac{1}{2}}I_r & 0 \\
        0 & TI_n
    \end{pmatrix}, \quad B_k = \begin{pmatrix}
        b_k & 0 \\
        0 & 0
    \end{pmatrix}, \quad B_w = \begin{pmatrix}
        b_w & 0 \\
        0 & 0
    \end{pmatrix}
    \]
    where \(b_k\) is the asymptotic bias of \(\hat{\Gamma}_k\) for \(0\le k\le p-1\) and 0 otherwise. Let \(W_{T, i} = w_{T, i+1} - w_{T, i}\) for \(1\le i \le p-1\), \(W_{T, 0} = w_{T, 1}\), and \(W_{T, p} = 1 - w_{T, p}\) and define \(W_i\) analogously for \(w\) instead of \(w_T\) so that, by assumption, \(T(W_{T, i} - W_i)\) converges in probability to 0 for \(T\rightarrow\infty\). Furthermore, define the random matrices \(Z_0,..., Z_p\in\mathbb{R}^{p\times p}\) such that \(Z_0=0\), \(Z_k\) is the right-hand side of \eqref{eq: asym_under} for \(1\le k < r\), \(Z_r\) is the right-hand side of \eqref{eq: asym_true}, and \(Z_k\) is the right-hand side of \eqref{eq: asym_over} for \(r < k \le p\). It then follows from Theorems \ref{thm: true_rank}, \ref{thm: over_rank}, and \ref{thm: under_rank} along with the continuous mapping theorem that
    \[
    \left(\hat{\Gamma}_{w_T} - \Gamma - B_w\right)D = \sum_{k=0}^p W_{T, k}\left(\hat{\Gamma}_k - \Gamma - B_k\right)D \rightarrow_w \sum_{k=0}^p W_k Z_k
    \]
    for \(T\rightarrow\infty\). Upon rewriting the right-hand side of the above expression we obtain \eqref{eq: asym_weight}.
\end{proof}

\begin{appendix}

\section{Multiple Lags}
\label{sec: multi}

In this section we consider processes of higher order. Let \(d\ge 1\) and \(\{Y_t\}_{t=0}^{\infty}\subset \mathbb{R}^p\) be an AR(\(d\))-process. Similar to (\ref{eq: vecm}), the dynamics of \(Y_t\) can be expressed in VECM form by
\[
\Delta Y_t = \Pi Y_{t-1} + \sum_{i=1}^{d-1}\Psi_i\Delta Y_{t-i} + Z_t
\]
where \(\{Z_t\}_{t=0}^{\infty}\) is a sequence of i.i.d. copies of \(Z_0\) with 0 mean and finite fourth moment. Define the processes \(X_{0t}=\Delta Y_t\), \(X_{1t}=Y_{t-1}\), and \(X_{2t}=( \Delta Y_{t-1},..., \Delta Y_{t-d})\) as well as the new parameter \(\Psi = (\Psi_1,..., \Psi_{d-1})\). We can then rewrite the equation as
\[
X_{0t} = \Pi X_{1t} + \Psi X_{2t}
\]
Assumptions similar to Assumptions \ref{as: roots} and \ref{as: orth} are needed to ensure a cointegrated process. We assume that \(1 \le r < p\).

\begin{assumption} \label{as: roots_ml}
The polynomial \(z\mapsto |(1 - z)I_p - \Pi z - \sum_{i=1}^{d-1}\Psi_i(1-z)z^i|\) has \(n = p - r\) unit roots and all other roots are outside the unit circle.
\end{assumption}

As before, this assumption implies that the rank of \(\Pi\) is \(p - n = r\) so that we can write \(\Pi = \alpha\beta^T\) for \(\alpha, \beta\in\mathbb{R}^{p\times r}\) of full rank \(r\).

\begin{assumption} \label{as: orth_ml}
The matrix \(\alpha_{\perp}^T (I_p - \sum_{i=1}^{d-1}\psi_i) \beta_{\perp}\) is non-singular.
\end{assumption}

The parameters are usually estimated as follows \citep{johansen1995likelihood}: First we find the residuals obtained from regressing \(X_{0t}\), \(X_{1t}\), and \(Z_t\) on \(X_{2t}\) denoted by \(R_{0t}\), \(R_{1t}\), and \(U_t\), respectively. The reduced rank estimator, \(\hat{\Pi}_k\) of \(\Pi\), is then obtained as above starting with the equation 
\[
R_{0t} = \Pi R_{1t} + U_t.
\]
After finding \(\hat{\Pi}_k\) an estimator for \(\Psi\) is given by ordinary least squares, i.e., \(\hat{\Psi}_{LS}\) is obtained by regressing \(X_{2t}\) on \(X_{0t} - \hat{\Pi}_k X_{1t}\). The asymptotics in this case can be derived from the previous section. Indeed, as shown in \cite{johansen1995likelihood}, similar limiting results as those given in Lemma \ref{lemma: cross} exist for the empirical cross-covariances given by \(R_{0t}\) and \(R_{1t}\) and the limiting behaviour of \(\hat{\Psi}_{LS}\) is studied in the usual way.

\section{Auxiliary results}
\label{app: aux}
We state here some results from perturbation theory of linear operators. These will be relevant especially for proving convergence of eigenvectors and eigenvalues. For more information, see~\cite{kato2013perturbation}. Let \(M, N \in\mathbb{R}^{p\times p}\) and denote by \(||\cdot||_F\) the Frobenius-norm. Define \(\rho(M, N)\in\mathbb{C}^p\) to be the ordered \(p\)-tuple that contains the solutions to \(M - \rho N = 0\) counted with multiplicity. 

\begin{lemma} \label{lemma: eig_val}
Assume that \(N\) is non-singular. Then the map \((M, N)\mapsto \rho(M, N)\) is continuous in the sense that for a sequence \(M_n, N_n\in\mathbb{R}^{p\times p}\) with \(||M-M_n||_F + ||N- N_n||_F\rightarrow 0\) it holds that
\[
||\rho(M, N) - \rho(M_n, N_N)|| \rightarrow 0.
\]
\end{lemma} 

\begin{proof}
This is Theorem 5.14 in~\cite{kato2013perturbation} after observing that for \(\rho\in\rho(M, N)\)
\[
|N^{-1}M - I\rho| = 0
\]
\end{proof}

If \(M\) and \(N\) are real symmetric, then \(\rho(M, N)\in\mathbb{R}^p\). Writing \(\rho(M, N)_i = \rho_i\) there then exist real-valued vectors \(v_1, ..., v_p\) satisfying \(Mv_i = \rho_i Nv_i\) and \(v_i^T N v_i=\delta_{ij}\) for \(i,j=1,..., p\). We call \(v_i\) the generalized eigenvector corresponding to \(\rho_i\). The generalized eigenvectors are not unique, but we can define \(P_i(M, N):=\sum_{j:\rho_i=\rho_j}v_iv_i^T\) for \(i=1,...,p\) which is then uniquely determined by \(M, N\). Note that for \(\rho_i = \rho_j\) we will also have \(P_i = P_j\). If we denote by \(S\) the space of real symmetric \(p\times p\) matrices, we are led to define the maps from \(S\times S\) to \(\mathbb{R}^{p\times p}\) given by \(P_i\) for \(i=1,..., p\). The following two lemmas concern the smoothness of these maps.

\begin{lemma} \label{lemma: eig_vec}
Let \(S_+\) denote the space of real positive definite \(p\times p\) matrices. \(P_i\) is continuous at points \((M, N)\in S\times S_+\). This means that if \((M_n, N_n)\in \mathbb{R}^{p\times p}\times \mathbb{R}^{p\times p}\) with \(||M_n - M||_F + ||N_n - N||_F\rightarrow 0\) for \(n\rightarrow \infty\), then
\[
||P_i(M_n, N_n) - P_i(M, N)||_F \rightarrow 0.
\]
\end{lemma}

\begin{proof}
We may assume for \(n\) large enough that \(N_n\) is positive definite since it converges towards a positive definite matrix. Define \(S = N^{-\frac{1}{2}}M N^{-\frac{1}{2}}\) and \(S_n = N_n^{-\frac{1}{2}}M_n N_n^{-\frac{1}{2}}\). Clearly the eigenvalues of \(S\) are \(\rho_1,..., \rho_p\) with the corresponding orthonormal eigenvectors given by \(\tilde{v}_i = N^{\frac{1}{2}}v_i\) for \(i=1, ..., p\) and similarly for \(S_n\). The result then follows from Theorem 2.23 and 3.16 in~\cite{kato2013perturbation}.
\end{proof}

\begin{lemma} \label{lemma: eig_diff}
Assume \((M, N)\in S \times S_+\) with simple eigenvalues, i.e., \(\rho_1 >...> \rho_p\). Then \(P_i\) and \(\rho_i\) are continously differentiable at \((M, N)\). Furthermore, the differential of \(P_i\) is given by
\[
dP_i = - P_i(dN)P_i - (M - \rho_i N)^+(dM - \rho_i dN)P_i - P_i(dM - \rho_i dN)(M - \rho_i)^+
\]
where \((M - \rho_i N)^+ = \sum_{j:\rho_j\neq\rho_i}(\rho_j - \rho_i)^{-1}P_j\).
\end{lemma}

\begin{proof}
The first statement follows directly by Theorem 8.9 in \cite{magnus2019matrix} after transforming the problem as above. To find the expression for the differential, we start with the defining equations
\[
MP_i = \rho_i NP_i, \quad \quad P_i N P_i = P_i.
 \]
The first equation yields \((M - \rho_i N)dP_i = -(dM - \rho_i dN)P_i + (d\rho_i)NP_i\). Note that \(P_jNP_i = \delta_{ij}P_i\) and that \(M-\rho_iN = (\sum_{j=1}^p (\rho_j - \rho_i) N P_j)N\). Multiplying the above differential equation by \((M - \rho_iN)^+\) on each side therefore yields
\begin{align}
    (I_p - P_i N)dP_i &= - (M - \rho_i N)^+(dM - \rho_i dN)P_i, \label{eq: diff_1}\\
    (dP_i)(I_p - NP_i) &= - P_i(dM - \rho_i dN)(M - \rho_i N)^+. \label{eq: diff_2}
\end{align}
Now after introducing differentials into the equation \(P_iNP_i=P_i\) we arrive at \((I_p - P_iN)dP_i + (dP_i)(I_p - NP_i) = dP_i + P_i(dN)P_i\). Plugging this into (\ref{eq: diff_1}) + (\ref{eq: diff_2}) gives us exactly the equation stated in the Lemma.
\end{proof}

It is not too hard to show that 
\[
\sum_{i=1}^p(M - \rho_i N)^+(dM - \rho_i dN)P_i = - \sum_{i=1}^p P_i(dM - \rho_i dN)(M - \rho_i N)^+
\]
and from the previous Lemma we then obtain for \(P = \sum_{i=1}^p P_i\) that 
\[
dP = -\sum_{i=1}^p P_i(dN)P_i.
\] 
But we also have that \(P= N^{-1}\) from which it follows that \(dP = - P(dN)P\) and thus \(-\sum_{i=1}^p P_i(dN)P_i= -P(dN)P\) and so the differential of \(P\) simplifies significantly.

\section{Simulation study}
\label{app: sim}
Here we describe in detail the simulation experiments from Section \ref{sec: simul}. The simulations were run in Python 3.9.1 and the code can be found on GitHub.\footnote{\url{https://github.com/cholberg/coint_CLT}}

\subsection{Comparison of Estimators} \label{app: sim_est}

We consider an array of parameters \(\Gamma_c\in\mathbb{R}^{p\times p}\) for \(p=3, 6, 9\) and \(c\in\{0, 0.2, 0.4,..., 29.8, 30\}\). For each \(c\) and \(p\), we let \(\Gamma_c\) be the diagonal matrix given by
\[
\left(\Gamma_c\right)_{ii} = \begin{cases}
    -1.5\quad &\text{if } i\le p/3, \\
    -c/T\quad &\text{if } p/3 < i \le 2p/3 \\
    0\quad &\text{otherwise. }
\end{cases}
\]
Throughout we fix \(T=100\). For each \(\Gamma_c\) we draw \(X_1,..., X_T\) and \(X_{T+1}\) and compute the MSPE across 4 million simulations. The results are given in Figure \ref{fig: mse_3_jag}.

\begin{figure}
    \centering
    \includegraphics[width=1\textwidth]{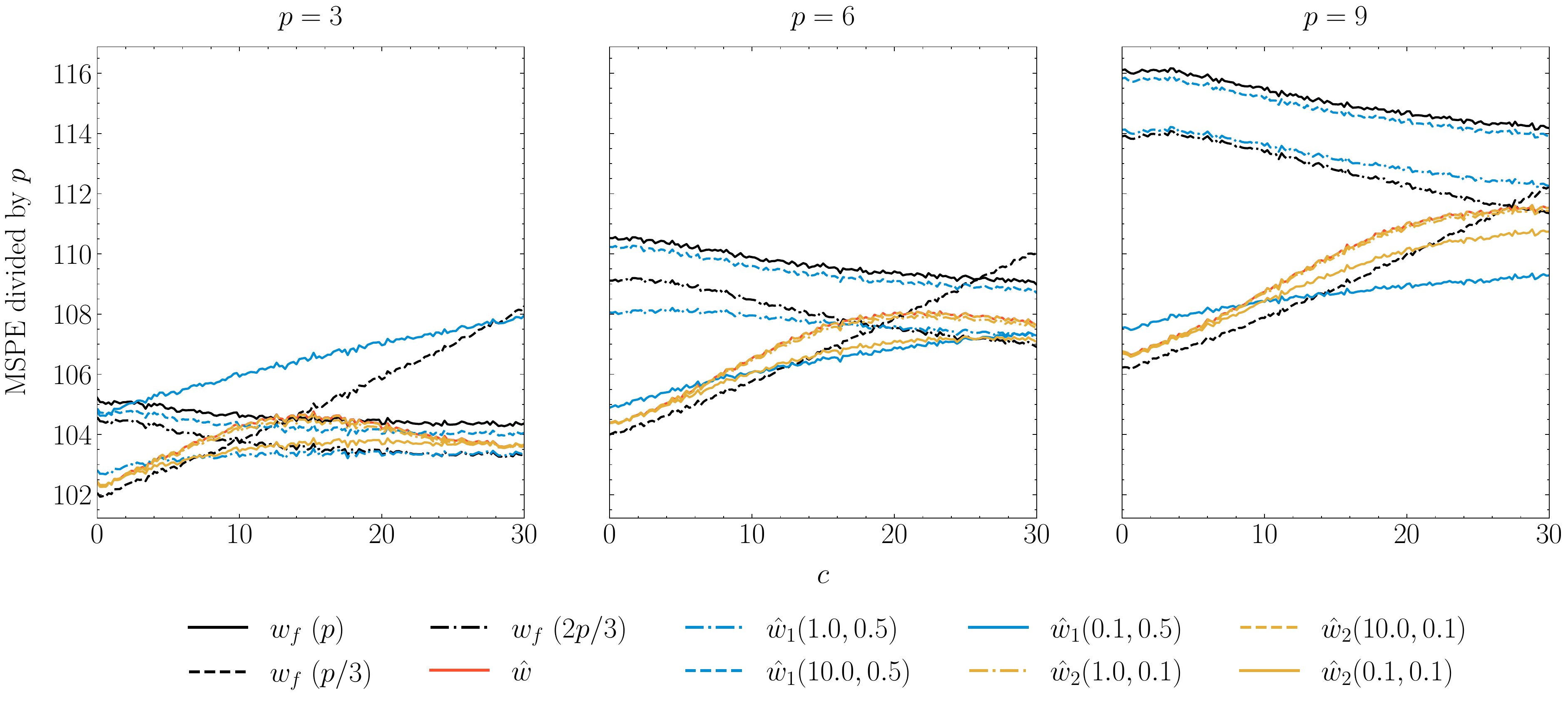}
    \caption{\small Mean square prediction error (MSPE) of different weighted reduced rank estimators for varying dimensions and \(c\in[0, 30]\) where the underlying autoregressive matrix, \(\Gamma_c\), has a third of its eigenvalues set to \(-c/T\), a third set to 0 and a third set to \(-3/2\). Sample size is fixed at \(T=100\).}
    \label{fig: mse_3_jag}
\end{figure}

To gain some insight we also plotted the mean and standard deviation of the individual weights across all simulations for \(p=3\) (see Figures \ref{fig: w_mean_3} and \ref{fig: w_std_3}). Clearly, choosing smoother weights significantly reduces the variance of the weights for eigenvectors where the corresponding eigenvalue is small. This is particularly apparent for \(\hat{w}_2(0.1, 0.1)\). While on average the weight is similar to \(\hat{w}\), it is not as steep for increasing \(c\) and its variance is significantly lower.

\begin{figure}
    \centering
    \includegraphics[width=1\textwidth]{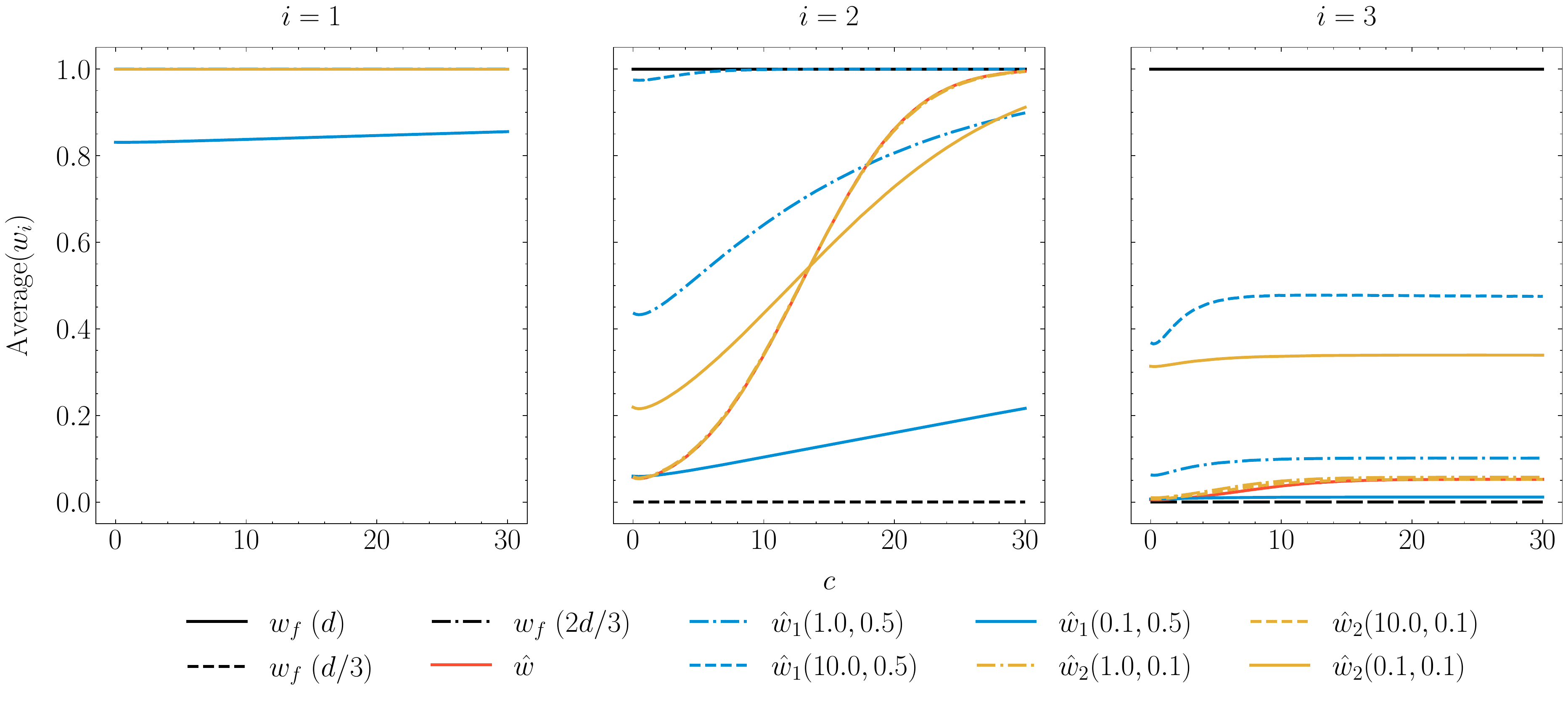}
    \caption{\small Mean of weights across 4 million simulations for \(d=3\) and \(c\in[0, 1]\) where the underlying autoregressive matrix, \(\Gamma_c\), has a third of its eigenvalues set to \(-c/T\), a third set to 0 and a third set to \(-3/2\). Sample size is fixed at \(T=100\).}
    \label{fig: w_mean_3}
\end{figure}
\begin{figure}
    \centering
    \includegraphics[width=1\textwidth]{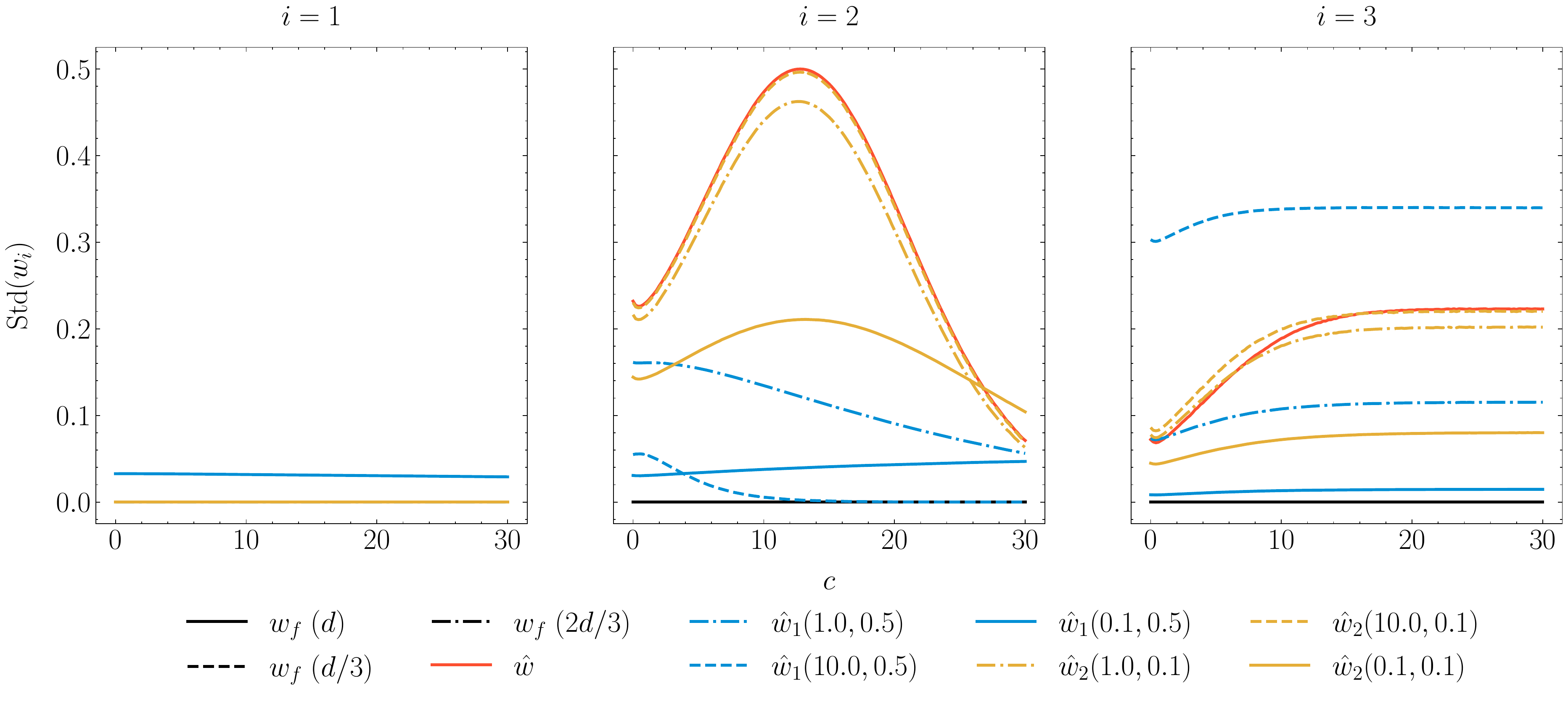}
    \caption{\small Standard deviation of weights across 4 million simulations for \(d=3\) and \(c\in[0, 1]\) where the underlying autoregressive matrix, \(\Gamma_c\), has a third of its eigenvalues set to \(-c/T\), a third set to 0 and a third set to \(-3/2\). Sample size is fixed at \(T=100\).}
    \label{fig: w_std_3}
\end{figure}

\subsection{Comparison of Distributions}
We generated 1000 i.i.d. samples of length \(T=5000\) of the process \(Y_t\in\mathbb{R}^4\) given by (\ref{eq: vecm}). For each sample, the errors are i.i.d. \(Z_t\sim\mathcal{N}(0, \Sigma_Z)\) where \(\Sigma_Z\) is a random positive definite matrix that is given by
\[
I_4 + \frac{1}{2}U U^T
\]
with \(U_{ij}\) i.i.d. uniformly over \([0, 1]\) for \(i,j=1,...,4\). The matrix \(\Pi\in\mathbb{R}^{4\times 4}\) is of rank 2 and can be decomposed as \(\Pi = \alpha\beta^T\) where 
\[
\alpha=\begin{pmatrix}-0.7 & 0 \\ 0 & -0.7 \\ 0 & 0 \\ 0 & 0\end{pmatrix}, \quad
\beta = \begin{pmatrix}1 & 0 \\ -1 & 1 \\ 0 & -1 \\ 0 & 0 \end{pmatrix}.
\]
This ensures that Assumptions \ref{as: roots}, \ref{as: orth}, and \ref{as: eig} are fulfilled with probability 1 so that the process is \(I(1)\) and cointegrated. The cointegration rank is equal to the rank of \(\Pi\) and the cointegration relations are given by the columns of \(\beta\). Each sample is \(Q\)-transformed to get \(X_t\) as in (\ref{eq: qvecm}) and we computed the estimators \(\hat{\Gamma}_1\), \(\hat{\Gamma}_2\) and \(\hat{\Gamma}_4\) to obtain the empirical large-sample distribution in the three different cases.

The asymptotic densities were obtained by generating 1000 samples from (\ref{eq: asym_true}), (\ref{eq: asym_over}), and (\ref{eq: asym_under}) with the parameters as given above.

\subsection{Rank selection vs. Bias} \label{app: rank_bias}

We now turn to the relation between the $r$ non-zero eigenvalues in (\ref{eq: population_eig}) and rank-selection. We consider a high-dimensional process \(Y_t\in\mathbb{R}^{40}\) generated by (\ref{eq: vecm}) under different parameter settings. The parameters are chosen in such a way that \(||\Pi||_F\) remains fixed in all settings with cointegration rank \(r=20\). However, the sequence of eigenvalues changes. For each setting, we consider different values of \(\lambda_{min}\) and \(\lambda_{max}\) such that \(\lambda_{min} < \lambda_{max}\) correspond to the smallest and largest squared eigenvalue, respectively. In order to keep \(||\Pi||_F\) fixed, an increase in \(\lambda_{min}\) leads to a decrease in \(\lambda_{max}\) so that it suffices to only specify the value of \(\lambda_{min}\). Smaller values of \(\lambda_{min}\) represent cases in which there are many small eigenvalues so that we would expect the cointegration test to be more prone to underestimate the true rank. To be precise for a given choice of \(\lambda_{min} \in \{0.01, 0.03, 0.1, 0.3\}\) the samples are generated as follows: For each sample the errors are i.i.d. \(Z_t\sim \mathcal{N}(0, I_{40})\). We fix
\[
\beta = \begin{pmatrix}
    I_{20} \\
    0_{20\times 20}
\end{pmatrix}
\]
and let \(\lambda_{max} = 0.81 - \lambda_{min}\). This choice may seem arbitrary, but it basically ensures that the process is non-explosive and that the norm of \(\Pi\) does not depend on \(\lambda_{min}\) as argued below.  Now define
\[
\lambda_k = \lambda_{min} + \frac{(\lambda_{max} - \lambda_{min}) (k - 1)}{19}, \quad \text{for } k = 1,..., 20.
\]
Setting \(D = \text{diag}(\sqrt{\lambda_1},..., \sqrt{\lambda_{20}})\) and letting
\[
\alpha = \begin{pmatrix}
    -2 D \\
    0_{20\times 20}
\end{pmatrix}
\]
then ensures that Assumptions \ref{as: roots}, \ref{as: orth}, and \ref{as: eig} are fulfilled with cointegration rank 20 and the non-zero eigenvalues defined in \eqref{eq: population_eig} being exactly \(\sqrt{\lambda_1},..., \sqrt{\lambda_{20}}\). Here, of course, \(\Pi = \alpha\beta^T\) and \(\Sigma_Z = I_{40}\). Observe that the process \(Y_t\) is already split up into its stationary and random walk part. By construction we also have that 
\[
||\Pi||_F^2 = \sum_{i=k}^{20}\lambda_k = 20\lambda_{min} +\frac{\lambda_{max}-\lambda_{min}}{19}\sum_{k=1}^{20}(k-1)=10(\lambda_{min}+\lambda_{max})=8.1
\]
so that \(||\Pi||_F\) does not depend on our choice of \(\lambda_{min}\).

Since the asymptotic distributions of the test statistics described in Section \ref{sec: rank} are non-standard and non-applicable in dimensions that much exceed \(p=12\) we will instead use a bootstrap approach as described in \cite{cavaliere2015bootstrap} to estimate the rank. In all simulations we used \(B=299\) bootstrap samples for each test.

In Fig. \ref{fig: rank_sel_high} we estimate the probability of choosing a given rank for different values of \(\lambda_{min}\). For each choice, we simulate \(200\) samples of \(Y_t\) of length \(T=200\) and estimate the rank using the sequential testing approach described above. We used the trace statistic and bootstrap to determine the asymptotic distribution under each hypothesis. In line with our expectations, the rank tends to be underestimated for smaller values of \(\lambda_{min}\). It appears that an increase in \(\lambda_{min}\) causes a shift towards the true rank in the distribution of estimated ranks. 

\begin{figure}
    \centering
    \includegraphics[width=1\textwidth]{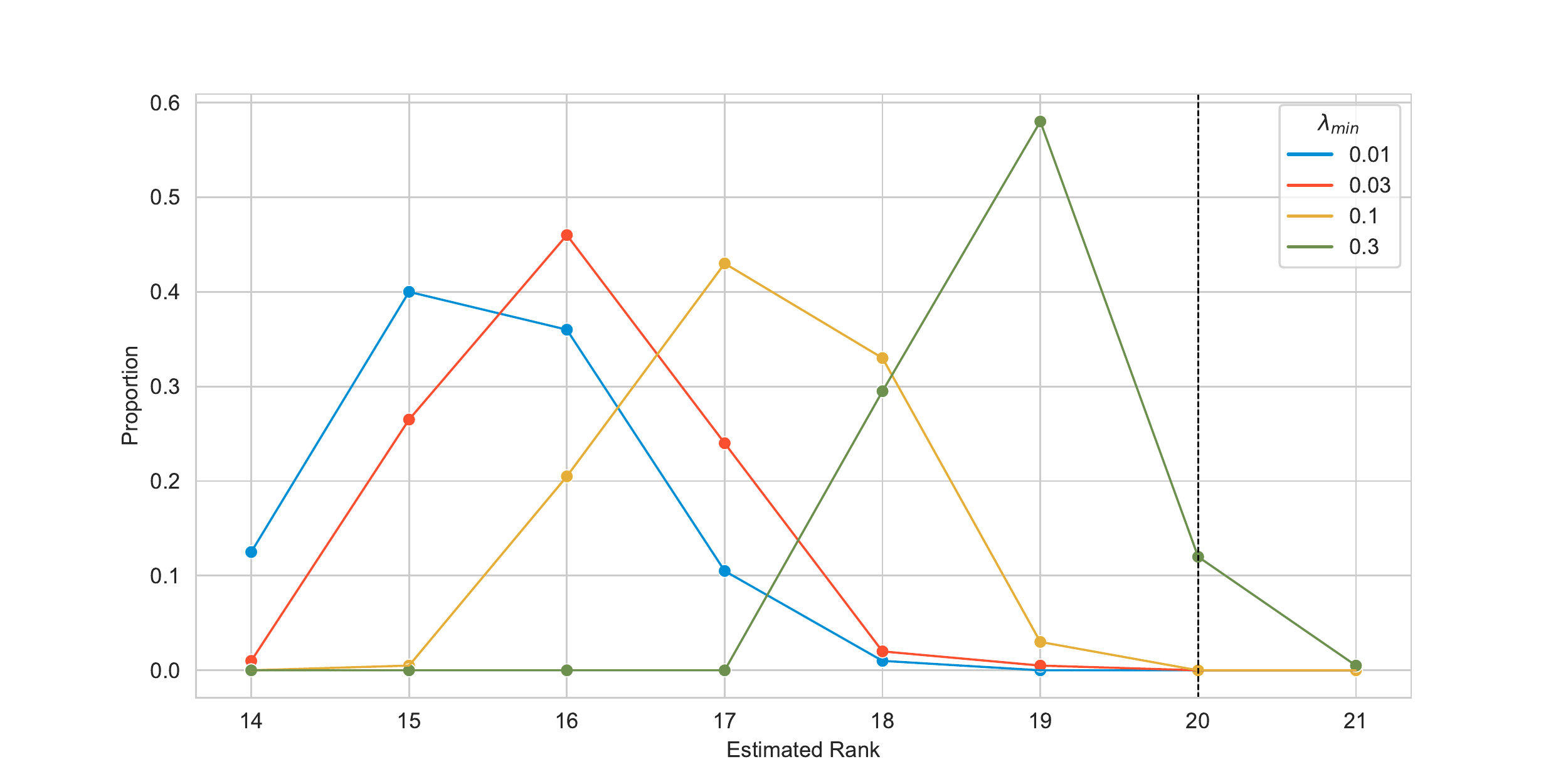}
    \caption{\small Distribution of the estimated rank under different parameter settings. The results are based on \(200\) simulations of \((Y_0,..., Y_{200})\). The vertical line at 20 illustrates the true rank of \(\Pi\). The dimension is $p=40$.}
    \label{fig: rank_sel_high}
\end{figure}

When studying the asymptotic bias in the different cases, we see an adverse effect. For smaller values of \(\lambda_{min}\), the bias tends to be very small as long as we only underestimate the true rank by a little. This is illustrated in Fig. \ref{fig: rank_bias_high}. Choosing \(k=14\) in the case where \(\lambda_{min}=0.01\) leads to approximately the same asymptotic bias as choosing \(k=18\) in the case where \(\lambda_{min}=0.3\) even though in the former case we are quite far off from the true rank.

\begin{figure}
    \centering
    \includegraphics[width=.9\textwidth]{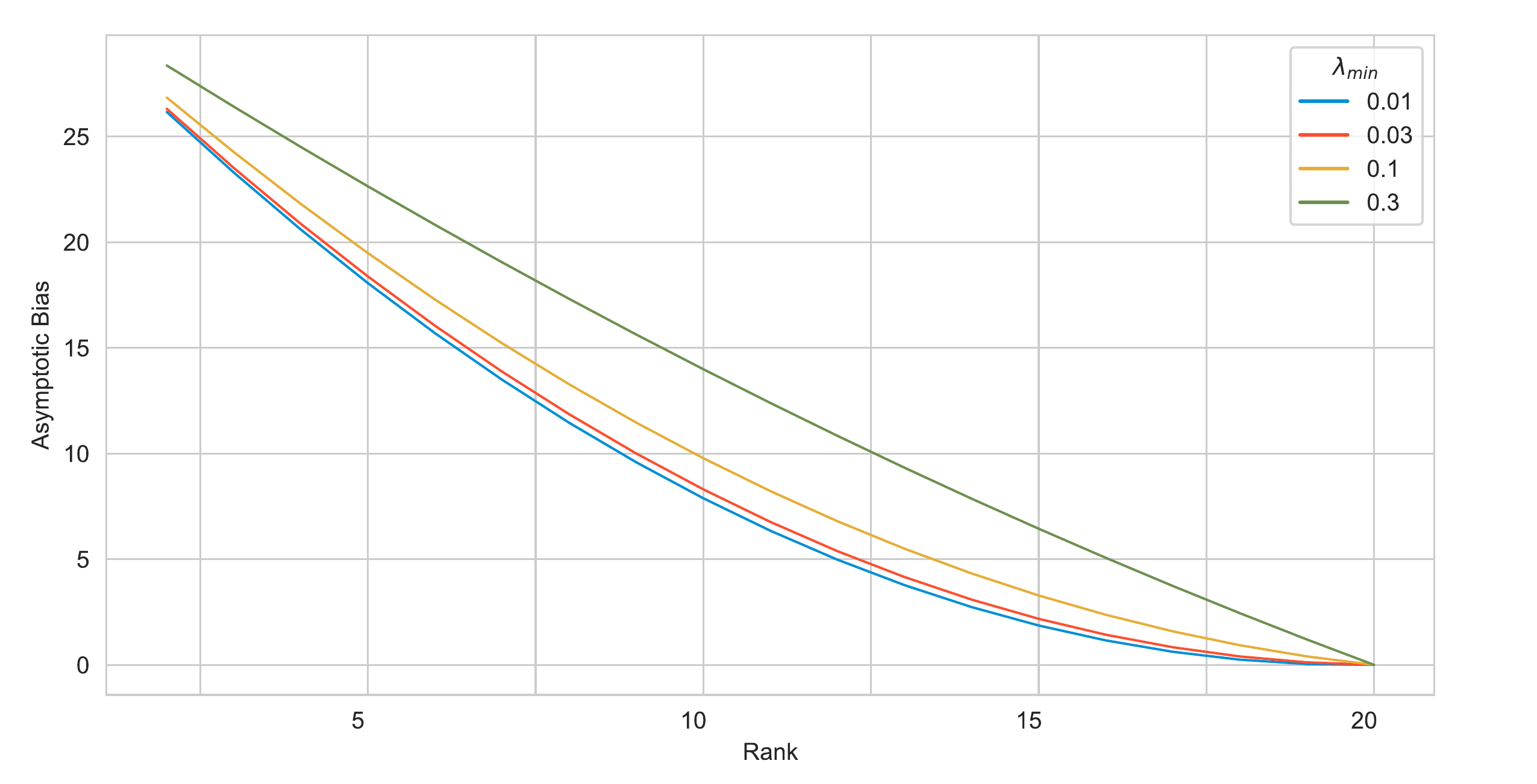}
    \caption{\small Asymptotic bias of \(\hat{\Pi}_k\) for different values of \(k\in\{2,..., 20\}\) and choices of \(\lambda_{min}\in\{0.01, 0.03, 0.1, 0.3\}\). We plot here the Frobenius norm of the bias \(b\) defined in \eqref{eq: b}.} 
    \label{fig: rank_bias_high}
\end{figure}

\end{appendix}

\bibliographystyle{plainnat}
\bibliography{bibliography}

\end{document}